\documentclass[twoside,11pt]{amsart}
\usepackage{mabliautoref}
\usepackage{amsmath,amsfonts,amssymb,amsthm,mathtools,mathrsfs}
\usepackage[a4paper,margin=1.1in]{geometry}  
\usepackage{lmodern}
\usepackage[utf8]{inputenc}
\usepackage[T1]{fontenc}
\usepackage[shortlabels]{enumitem}

\RequirePackage[dvipsnames]{xcolor}
\usepackage{tikz-cd}
\usepackage{chngcntr}
\usepackage{fdsymbol}
\usepackage{minibox}
\usepackage[framemethod=TikZ]{mdframed}
\usepackage{ifthen}
\usepackage{relsize}
\usepackage[most]{tcolorbox}
\usepackage{etoolbox}

\hypersetup{
bookmarksdepth=3,
bookmarksopen,
bookmarksnumbered,
pdfstartview=FitH,
colorlinks,%backref,hyperindex,
linkcolor=Sepia,
anchorcolor=BurntOrange,
citecolor=OliveGreen,
filecolor=BlueViolet,
menucolor=Yellow,
urlcolor=OliveGreen
}

\makeatletter
\renewcommand\subsection{
  \renewcommand{\sfdefault}{pag}
  \@startsection{subsection}%
  {2}{0pt}{.8\baselineskip}{.4\baselineskip}{\raggedright
 \sffamily\itshape\small\bfseries
  }}
\renewcommand\section{
  \renewcommand{\sfdefault}{phv}
  \@startsection{section} %
  {1}{0pt}{\baselineskip}{.8\baselineskip}{\centering
    \sffamily
    \scshape
    \bfseries
}}
\makeatother

\usetikzlibrary{decorations.markings}
\tikzset{degil/.style={
  decoration={markings,
  mark= at position 0.5 with {
  \node[transform shape] (tempnode) {$\backslash$};
  \draw[thick] (tempnode.north east) -- (tempnode.south west);
  }}, postaction={decorate}
}}
\usetikzlibrary{decorations.text, calc, matrix, arrows, decorations.pathreplacing}
\usetikzlibrary{shapes.geometric, arrows, positioning, fit, calc,decorations.pathmorphing,shapes}
\tikzstyle{e} = [ellipse, minimum width=1cm, minimum height=0.5cm,text centered, draw=black]%, fill=red!30]
\tikzstyle{arrow} = [thick,-,>=stealth]
\tikzstyle{connection}=[inner sep=0,outer sep=0]

\tikzset{
    vertl/.style={anchor=south, rotate=90, inner sep=.4mm, pos=0.4}
}

\tikzset{
    vertr/.style={anchor=north, rotate=90, inner sep=.7mm, pos=0.4}
}

\tikzset{
    diag/.style={anchor=south, rotate=30, inner sep=.5mm}
}

\tikzset{
    ddiag/.style={anchor=south, rotate=-16, inner sep=0mm}
}

\newcommand{\bA}{\mathbb{A}}
\newcommand{\bC}{\mathbb{C}}

\newcommand{\bG}{\mathbb{G}}
\newcommand{\bK}{\mathbb{K}}

\newcommand{\bN}{\mathbb{N}}
\newcommand{\bP}{\mathbb{P}}
\newcommand{\bQ}{\mathbb{Q}}

\newcommand{\bT}{\mathbb{T}}
\newcommand{\bZ}{\mathbb{Z}}

\newcommand{\cL}{\mathcal{L}}

\newcommand{\cO}{\mathcal{O}}

\newcommand{\sZ}{\mathscr{Z}}

\newcommand{\fm}{\mathfrak{m}}

\newcommand{\tU}{\widetilde{U}}

\newcommand{\tX}{\widetilde{X}}

\DeclareMathOperator{\Alb}{Alb}

\DeclareMathOperator{\codim}{codim}
\DeclareMathOperator{\Pic}{Pic}

\DeclareMathOperator{\red}{red}

\DeclareMathOperator{\Spec}{Spec}

\DeclareMathOperator{\Diff}{Diff}

\DeclareMathOperator{\ord}{ord}
\DeclareMathOperator{\Supp}{Supp}

\DeclareMathOperator{\GL}{{GL}}
\DeclareMathOperator{\Isom}{Isom}
\DeclareMathOperator{\pr}{{pr}}
\DeclareMathOperator{\alb}{{alb}}
\DeclareMathOperator{\dlt}{{dlt}}
\DeclareMathOperator{\sm}{sm}
\DeclareMathOperator{\id}{{id}}
\DeclareMathOperator{\Aut}{Aut}
\DeclareMathOperator{\Bl}{Bl}

\newcommand{\expl}[2]{\underset{\mathclap{\minibox[c]{$\uparrow$\\ \fbox{\footnotesize #2}}}}{#1}}
\newcommand{\explshift}[3]{\underset{\mathclap{\minibox[c]{$\uparrow$\\ \hspace{#1} \fbox{\footnotesize #3}}}}{#2}}

\newcommand{\stacksproj}[1]{{\cite[Tag~\href{http://stacks.math.columbia.edu/tag/#1}{#1}]{stacks-project}}}

\newcommand{\wt}{\widetilde}
\newcommand{\Samuel}[2]{
\ifthenelse{\equal{#2}{no}}
  {% true case
   \nu_{\fm\mid#1}
  }
  {% false case
   \nu_{\fm\mid I}(#2)
  }
}
\newcommand{\Samuell}[3]{
\ifthenelse{\equal{#3}{no}}
  {% true case
   \nu_{#1\mid#2}
  }
  {% false case
   \nu_{#1\mid#2}(#3)
  }
}
\newcommand{\aSamuel}[2]{
\ifthenelse{\equal{#2}{no}}
  {% true case
   \mathop{\overline{\nu}_{\fm\mid#1}}
  }
  {% false case
   \mathop{\overline{\nu}_{\fm\mid#1}}(#2)
  }
}
\newcommand{\aSamuell}[3]{
\ifthenelse{\equal{#3}{no}}
  {% true case
   \mathop{\overline{\nu}_{#1\mid#2}}
  }
  {% false case
   \mathop{\overline{\nu}_{#1\mid#2}}(#3)
  }
}
\newcommand{\Ard}[2]{
\ifthenelse{\equal{#2}{no}}
  {% true case
   \mathop{\overline{\ord}_{#1}}
  }
  {% false case
   \mathop{\overline{\ord}_{#1}}(#2)
  }
}
\newcommand{\resord}[3]{
\ifthenelse{\equal{#3}{no}}
  {% true case
   \ord_{#1}^{|#2}
  }
  {% false case
   \ord_{#1}^{|#2}(#3)
  }
}
\newcommand{\aresord}[3]{
\ifthenelse{\equal{#3}{no}}
  {% true case
   \mathop{\overline{\ord}_{#1}^{|#2}}
  }
  {% false case
   \mathop{\overline{\ord}_{#1}^{|#2}}(#3)
  }
}
\newcommand{\Mult}[3]{
\ifthenelse{\equal{#3}{no}}
  {% true case
   \overline{\operatorname{mult}}_{#1}#2
  }
  {% false case
   \ifthenelse{\equal{#2}{no}}
  {% true case
   \overline{\operatorname{mult}}_{#1}\bullet|_{#3}
  }
  {% false case
   \overline{\operatorname{mult}}_{#1}#2|_{#3}
  }
  }
}

\newcommand{\factor}[2]{\left. \raise 2pt\hbox{\ensuremath{#1}} \right/
        \hskip -2pt\raise -2pt\hbox{\ensuremath{#2}}}

\newcommand{\acts}{\medspace\rotatebox[origin=c]{-90}{\raisebox{1.5pt}{\scalebox{0.75}{$\circlearrowright$}}}\medspace}
\newcommand{\catname}[1]{{\normalfont\textbf{#1}}}
\newcommand{\Sch}{\catname{Sch}}
\newcommand{\Set}{\catname{Set}}
\newcommand{\GSch}{\bT\catname{-Sch}}
\newcommand*{\qedclaim}{
  \null\nobreak\hfill\ensuremath{\blacksquare}
}

\makeatletter
\newcommand{\specialcell}[1]{\ifmeasuring@#1\else\omit$\displaystyle{#1}$\ignorespaces\fi}
\makeatother

\counterwithin*{equation}{section}
\counterwithin*{equation}{subsection}

\AtBeginDocument{
   \def\MR#1{}
}

\setlist[itemize,enumerate]{
  leftmargin=0.9cm,
  topsep=2pt,        % space before/after list
  itemsep=1pt,       % space between items
  parsep=0pt,        % space between paragraphs within an item
  partopsep=0pt      % extra space when list starts a paragraph
}
\renewcommand{\theenumi}{\arabic{enumi}}
\mdfsetup{
 skipabove=6pt,
 skipbelow=3pt,
 roundcorner=10pt
}

\makeatletter
\newcommand{\fillqed}[2][0pt]{% #1 = optional extra shift (e.g., -2pt), #2 = equation
  \settowidth{\@tempdima}{$\displaystyle #2$}%          width of equation
  \settowidth{\@tempdimb}{\qedsymbol}%                  width of QED symbol
  \@tempdimc = \linewidth
  \advance\@tempdimc by -\@tempdima
  \advance\@tempdimc by -\@tempdimb
  \divide\@tempdimc by 2                                 % half-skip
  \hskip\@tempdimc\relax #2\hskip\@tempdimc\relax
  \kern#1\relax \qedhere                                 % apply correction
}
\makeatother
\author[L. Rösler]{Linus Rösler}
  \address{\'Ecole Polytechnique F\'ed\'erale de Lausanne, Chair of Algebraic Geometry \newline 
    \indent MA C3 615 (Bâtiment MA), Station 8, CH-1015 Lausanne}
  \email{linus.rosler@epfl.ch}

  \title[On the birational isotriviality of the Albanese morphism]{On the birational isotriviality of the Albanese morphism of a log Calabi--Yau pair with a torus action}

    \AtBeginDocument{%
   \def\MR#1{}
    }

\begin{document}

\begin{abstract}
    Let $(X,\Delta)$ be a projective, log canonical, $K$--trivial pair over the complex numbers. Let $Z$ be a minimal log canonical center of $(X,\Delta)$ and suppose that there exists a torus $\bT\subseteq\Aut(X)$ preserving $\Delta$ and such that $\dim\bT=\codim_X Z$. Then we show that two general fibers of the Albanese morphism $\alb_X$ are birationally equivalent. In particular, the pathological example of a projective, log canonical, $K$--trivial variety whose Albanese morphism is not generically birationally isotrivial, recently constructed by Bernasconi, Filipazzi, Patakfalvi and Tsakanikas, can be avoided under the additional hypothesis that there exists a torus of large enough dimension in the automorphism group of the given pair.
\end{abstract} 

\maketitle

%======================================================================================================
%======================================================================================================

\tableofcontents

\section{Introduction}\label{sec:Albanese_intro}
We work over the field of complex numbers $\bC$.

In the birational classification of projective varieties, normal projective varieties with mild singularities and torsion canonical sheaf (in short, $K$--trivial) form an important class. Indeed, the Minimal Model Program predicts that they are one of the three fundamental building blocks into which every projective variety can be decomposed through iterated fibrations, up to birational equivalence. In a next step, among $K$--trivial varieties, smooth projective varieties which admit an algebraic group structure form a nice and well--understood subclass: abelian varieties. One reason for their general importance is that every smooth projective variety admits a universal morphism to an abelian variety, called the Albanese morphism. This provides a powerful tool in birational geometry, as thereby every smooth projective variety can be examined through the lens of an abelian variety. For a smooth projective $K$--trivial variety $X$, this point of view is of special importance, as its Albanese morphism exhibits especially nice properties: it is a fibration and all of its closed fibers are isomorphic to one another (\cite{Calabi_On_Kahler_manifolds_with_vanishing_canonical_class}). Hence, it provides a decomposition of $X$ into an abelian variety and the fiber. So in the case where the Albanese variety of $X$ is non--trivial, we reduce the study of the geometry of $X$ to lower dimension. Following \cite{Bernasconi_Filipazzi_Patakfalvi_Tsakanikas_A_counterexample_to_the_log_canonical_Beauville-Bogomolov_decomposition}, we call this the weak Beauville--Bogomolov decomposition of $X$.

It is an active area of research to generalize the structure theory of $K$--trivial varieties, such as the weak Beauville--Bogomolov decomposition, in several directions: allow boundaries, mild singularities, and negativity in the canonical sheaf. Through the combined work in \cite{Greb_Kebekus_Petternel_Singular_Spaces_with_Trivial_Canonical_Class, Druel_A_decomposition_theorem_for_singular_spaces_with_trivial_canonical_class___of_dimension_at_most_five, Greb_Guenancia_Kebekus_Klt_varieties_with_trivial_canonical_class_holonomy_differential_forms_and_fundamental_groups, Horing_Peternell_Algebraic_integrability_of_foliations_with_numerically_trivial_canonical___bundle, Campana_Cao_Matsumura_Projective_klt_pairs_with_nef_anti-canonical_divisor, Patakfalvi_Zdanowicz_On_the_Beauville-Bogomolov_decomposition_in_characteristic_p, Matsumura_Wang_Structure_theorem_for_projective_klt_pairs_with_nef_anti-canonical_divisor}, the weak Beauville--Bogomolov decomposition has been established for projective Kawamata log terminal pairs with nef anti--canonical divisor (which includes $K$--trivial ones), see \cite[Theorem A]{Bernasconi_Filipazzi_Patakfalvi_Tsakanikas_A_counterexample_to_the_log_canonical_Beauville-Bogomolov_decomposition}. On the other hand, in \cite[Theorem 1.3]{Bernasconi_Filipazzi_Patakfalvi_Tsakanikas_A_counterexample_to_the_log_canonical_Beauville-Bogomolov_decomposition}, the authors show that it may strongly fail for $K$--trivial varieties with log canonical singularities. Indeed, they construct a projective variety $X$ with strictly log canonical singularities and trivial canonical bundle such that $\Alb_X$ is an elliptic curve and every fiber of $\alb_X$ is birational to at most finitely many other fibers (see also the recent work \cite[Theorem B]{Collins_Guenancia_Log_Calabi-Yau_manifolds_holomorphic_tensors_stability_and_universal_cover}). On the positive side, by \cite[Theorem 1.5]{Bernasconi_Filipazzi_Patakfalvi_Tsakanikas_A_counterexample_to_the_log_canonical_Beauville-Bogomolov_decomposition}, the Albanese morphism of a projective, log canonical pair with nef anti--canonical divisor is still a flat fibration with geometrically reduced fibers. We would also like to point out the recent work in \cite{Zhu_On_quasi-Albanese_morphisms_for_log_canonical_Calabi-Yau_pairs} about the quasi--Albanese morphism of a log Calabi--Yau pair.

The main purpose of this article is to obtain a positive result in the direction of the weak Beauville--Bogomolov decomposition for a projective, log canonical, $K$--trivial pair $(X,\Delta)$ (in short, a log Calabi--Yau pair) by imposing some additional hypotheses. We take an approach which is a common strategy for simplifying hard problems in algebraic geometry: we additionally assume that we have a non-trivial torus action $\bT\acts (X,\Delta)$. This gives us new angles to tackle the problem, as it allows the use of, e.g., combinatorial techniques. 

Given a log Calabi--Yau pair, suppose that there is a torus $\bT\subseteq\Aut(X)$ that preserves the boundary $\Delta$. The larger the dimension of $\bT$, the simpler the structure of $(X,\Delta)$ is; e.g., if $\dim\bT=\dim X$, then $X$ is toric (see \autoref{lem:complexity}). There is an elementary upper bound, for the dimension of such a torus, depending on the geometry of $(X,\Delta)$: if $Z$ is a minimal log canonical center of $(X,\Delta)$, then by \autoref{cor:dim_bound_torus_center} we have $\dim \bT\leq \codim_X Z$. From this point of view, the first question one should address is whether some form of the weak Beauville--Bogomolov decomposition holds under the hypothesis that there exists a torus $\bT$ as above achieving the maximal dimension, i.e., $\dim\bT=\codim_X Z$. Our main result is the following affirmative answer.

\begin{thmletter}\label{thmA}
    Let $(X,\Delta)$ be a log Calabi--Yau pair and let $Z$ be a minimal log canonical center of $(X,\Delta)$. Assume that there exists a torus $\bT\subseteq\Aut(X)$ of dimension $\codim_X Z$ that preserves $\Delta$. Then the Albanese morphism $\alb_X\colon X\to\Alb_X$ is generically birationally isotrivial.
\end{thmletter}

\begin{remintro}\label{rem:birational_isotriviality}
    The term `generically birationally isotrivial' means that two general closed fibers are birationally equivalent. For a flat fibration, by \cite[Theorem 1.1, Corollary 1.3]{Bogomolov_Bohning_Graf_von_Bothmer_Birationally_isotrivial_fiber_spaces} this is equivalent to requiring that, up to taking a finite cover of a dense open subset (by shrinking it we may assume that we have an étale open) of the base, the family is birational to the trivial family. Since for a log Calabi--Yau pair $(X,\Delta)$ the Albanese morphism $\alb_X$ is flat by \cite[Theorem 1.5]{Bernasconi_Filipazzi_Patakfalvi_Tsakanikas_A_counterexample_to_the_log_canonical_Beauville-Bogomolov_decomposition}, this provides an alternative formulation of \autoref{thmA}: there exists an étale open $U$ in $\Alb_X$ such that $X\times_{\Alb_X}U\simeq X_0\times U$, where $X_0$ is a birational model of the general fiber.
\end{remintro}

Therefore, in the situation where $(X,\Delta)$ has the simplest geometry from a toric perspective, we obtain a birational version of the weak Beauville--Bogomolov decomposition. In particular, pathological examples such as \cite[Theorem 1.3]{Bernasconi_Filipazzi_Patakfalvi_Tsakanikas_A_counterexample_to_the_log_canonical_Beauville-Bogomolov_decomposition} cannot occur.

Let us comment on the sharpness of the conclusion of \autoref{thmA}, or in other words, whether one can hope for more than generic birational isotriviality. It is fairly straightforward to construct an example which is not birationally isotrivial, see \autoref{ex:gen_isotr_not_isotr}. However, this example is still generically isotrivial, so one might ask whether this is always the case. The answer is no: in \autoref{ex:bir_isotr_not_gen_isotr}, with a similar but slightly more involved construction, we construct a pair that satisfies the hypotheses of \autoref{thmA}, but whose Albanese morphism is neither generically isotrivial, nor birationally isotrivial. So in this sense, \autoref{thmA} is sharp.

However, in the construction for \autoref{ex:bir_isotr_not_gen_isotr}, we broke the birational isotriviality in a somewhat trivial way, namely by creating a reducible fiber through a simple blow--up. This raises the following question.% (see \autoref{rem:strategy_birationally_isotrivial} for a possible strategy in the dlt case).

\begin{questionintro}\label{quest:bir_isotr}
    Let $(X,\Delta)$ be a pair satisfying the hypotheses of \autoref{thmA}. Is it always the case that for any two closed fibers $F,F'$ of $\alb_X$, there exists a component of $F$ which is birational to a component of $F'$?
\end{questionintro}

Furthermore, as the boundary $\Delta$ is missing from the conclusion in \autoref{thmA}, one might ask whether it can be incorporated in the birational isotriviality.

\begin{questionintro}\label{quest:crep_bir_isotr}
    Let $(X,\Delta)$ be a pair satisfying the hypotheses of \autoref{thmA}. Is $\alb_X\colon (X,\Delta)\to \Alb_X$ generically crepant birationally isotrivial? That is, for two general fibers $F,F'$ of $\alb_X$, are $(F,\Delta|_F)$ and $(F',\Delta|_{F'})$ crepant birational?
\end{questionintro}

Note that the Albanese morphism of the pair in \autoref{ex:bir_isotr_not_gen_isotr} is generically crepant birationally isotrivial. However, with our current method of proof, it is unclear how to incorporate the boundary in the birationality of the fibers, essentially because the Bia\l ynicki--Birula decomposition (see below resp.\ \autoref{subsec:Birula}) does not see the boundary.

\subsection{Strategy of proof}\label{subsec:Albanese_strategy}

Let $(X,\Delta)$ be a log Calabi--Yau pair with a minimal lc center $Z$ and a torus $\bT\subseteq\Aut(X)$ with $\dim\bT=\codim_X Z$. For simplicity, assume that $(X,\Delta)$ is dlt. By \cite[Theorem 4.19]{Kollar_Singularities_of_the_minimal_model_program}, we can perform adjunction for $Z$, namely there is a choice of boundary $\Delta_Z$ on $Z$ such that $(Z,\Delta_Z)$ is $K$--trivial and by minimality of $Z$ it is klt. In particular, since the weak Beauville--Bogomolov decomposition holds for $K$--trivial klt pairs (\cite{Patakfalvi_Zdanowicz_On_the_Beauville-Bogomolov_decomposition_in_characteristic_p,Matsumura_Wang_Structure_theorem_for_projective_klt_pairs_with_nef_anti-canonical_divisor}), we know that $\alb_Z$ is isotrivial and that $Z$ dominates $\Alb_X$, which is useful to compare the fibers of $\alb_Z$ with the fibers of $\alb_X$. The strategy is now to use the torus action to reduce the birational isotriviality of $\alb_X$ to the isotriviality of $\alb_Z$.

By \autoref{cor:dim_bound_torus_center} and \autoref{lem:dim_T_action}, $Z$ is a top--dimensional component of the fixed locus $X^{\bT}$. For torus actions, there is a powerful tool to relate $X$ and $X^{\bT}$: the Bia\l ynicki--Birula decomposition, originating in \cite{Bialynicki_Some_theorems_on_actions_of_algebraic_groups}. For a smooth projective variety $X$ with a $\bG_m$--action, there is a partition $X=X_1\sqcup\cdots\sqcup X_r$ of $X$ into locally closed subsets $X_i\subseteq X$ such that on each cell $X_i$, the $\bG_m$--action extends to a multiplicative action by $\bA^1=\bG_m\cup\{0\}$. In particular, we have in each cell a set of limit points $F_i=0\cdot X_i\subseteq X_i$, and a map $X_i\to F_i$ sending the points $x$ of $X_i$ to their limit $0\cdot x\in F_i$. As $0$ is a fixed point of the action $\bG_m\acts \bA^1$, we have $F_i\subseteq X^{\bG_m}$; in fact the $F_i$ are precisely the connected components of the fixed locus. Strikingly, the limit maps $X_i\to F_i$ are actually affine fiber bundles, see \cite[Theorem 4.1]{Bialynicki_Some_theorems_on_actions_of_algebraic_groups}.

In \cite{Jelisiejew_Sienkiewicz_Bialynicki-Birula_decomposition_for_reductive_groups}, a version of the Bia\l ynicki--Birula decomposition was established for actions of linearily reductive groups (so in particular higher--dimensional tori), for much more general schemes (locally of finite type over $\bC$ suffices), and with nice functorial properties. More precisely, for a torus action $\bT\acts X$, one considers a partial compactification $\bT\subseteq\overline{\bT}$ where $\overline{\bT}$ is a linearily reductive monoid (i.e. $\overline{\bT}$ is an affine variety with an associative multiplication $\overline{\bT}\times\overline{\bT}\to\overline{\bT}$ with unit, such that $\bT$ is the dense submonoid of invertible elements). This monoid $\overline{\bT}$ will play the role of the partial compactification $\bG_m\subseteq\bA^1$ in the original Bia\l ynicki--Birula decomposition. For every choice of $\overline{\bT}$, one obtains a Bia\l ynicki--Birula decomposition $i_X\colon X^+\to X$ such that on $X^+$ the torus action extends to a $\overline{\bT}$--action, together with a limit map $\pi_X\colon X^+\to X^{\bT}$. In the paragraph above about the original Bia\l yinicki--Birula decomposition, $X^+$ would be the disjoint union $X^+=X_1\sqcup\cdots\sqcup X_r$, and following the terminology in \cite{Jelisiejew_Sienkiewicz_Bialynicki-Birula_decomposition_for_reductive_groups}, call the components of $X^+$ \emph{cells}. Although this is obtained in great generality, it is only in the smooth case where $X^+$ is known to have similar properties as in the classical Bia\l ynicki--Birula decomposition, e.g., that $\pi_X$ is an affine fiber bundle.

Now note that because of \autoref{cor:dim_bound_torus_center}, $Z$ is a component of $X^{\bT}$. If we can make sure that the cell over $Z$ in the Bia\l ynicki--Birula decomposition is open in $X$, then we obtain an open subset of $X$ which is an affine open fiber bundle over an open subset of $Z$. The result \cite[Proposition 1.6]{Jelisiejew_Sienkiewicz_Bialynicki-Birula_decomposition_for_reductive_groups} provides a useful criterion: if for $z\in Z$ the action of $\bT$ on the cotangent space $T_z^{\vee}(X)$ extends to a $\overline{\bT}$--action, then $z$ lies in an open cell. Using the weight decomposition of $\bT\acts T_z^{\vee}(X)$, this reduces to linear algebra. In \autoref{lem:extend_lin_t_action}, we use the hypothesis on the dimension of $\bT$ to construct a partial compactification $\overline{\bT}$ for which this holds.

Then, in order to use the open cell to relate the Albanese morphisms of $X$ and $Z$, we use an elementary version of the Albanese morphism of open varieties, as constructed by Serre in \cite{Serre_Morphismes_universels_et_variete_dAlbanese}, see \autoref{subsec:Albanese} (note also the much more refined work on this in \cite{Fujino_On_quasi-Albanese_maps,Schroer_Albanese_maps_for_open_algebraic_spaces}). Essentially, the Albanese variety of an open variety is given by the Albanese variety of a big enough compactification, see \autoref{lem:alb_of_big_comp}. As we are working with open varieties, we will need non--proper versions of the Rigidity Lemma (\autoref{lem:rigidity_magic}) and the fact that a rationally chain connected variety has trivial Albanese (\autoref{lem:alb_iso_rcc}). Putting all of this together will yield the result in the dlt case (\autoref{prop:dlt_case_arbitrary_A}).

Finally, to pass from the dlt case to the lc case, we use a $\bT$--equivariant dlt modification (\autoref{lem:dlt_mod}). Its existence is based on functorial resolution of singularities (\cite[Theorem 3.45]{Kollar_Lectures_on_resolution_of_singularities}) and the fact that for a connected algebraic group $G$, any MMP on a pair $(X,\Delta)$ with a $G$--action is a $G$--MMP (\cite[Remark 2.3]{Floris_A_note_on_the_G-Sarkisov_program}).

\subsection{Notation}\label{subsec:Albanese_notation}
\begin{itemize}
    \item A \emph{variety} $X$ is an integral separated scheme of finite type over $\mathbb{C}$.

    \item For a normal variety $X$, we denote by $\omega_X$ its canonical sheaf. A \emph{canonical divisor} $K_X$ is a Weil divisor on $X$ such that $\mathcal{O}_X(K_X) = \omega_X$.
    
    \item A \emph{torus} is an algebraic group $\bT$ isomorphic to a non-negative power of the multiplicative group $\bG_m\coloneqq\Spec\bC[x,x^{-1}]$.
    
    \item A \emph{compactification} of a variety $U$ is an open immersion $i\colon U\to X$ where $X$ is projective. Another compactification $j\colon U\to Y$ is said to \emph{dominate} $i$ if there exists $f\colon Y\to X$ such that $f\circ j = i$.

    \item We say $(X,\Delta)$ is a \emph{pair} if $X$ is a normal variety and $\Delta$ is an effective $\mathbb{Q}$-divisor such that $K_X + \Delta$ is $\mathbb{Q}$--Cartier.
    
    \item We say that two pairs $(X_1,\Delta_1)$ and $(X_2,\Delta_2)$ are \emph{crepant birational to one another} if there exist proper birational morphisms $p_1\colon Y\to X_1$ and $p_2\colon Y\to X_2$ from a normal variety $Y$ such that $p_1^*(K_{X_1} + \Delta_1) = p_2^*(K_{X_2} + \Delta_2)$.

    \item Let $f \colon X \to Y$ be a flat projective morphism of varieties. We say that $f$ is \emph{isotrivial} if it is locally trivial (i.e., isomorphic to the product family) in the \'etale topology. We say that it is generically isotrivial if it is isotrivial over a Zariski open subset. Note that if $f$ is (generically) isotrivial then two (general) closed fibers of $f$ are isomorphic.
    
    \item Let $f\colon X\to Y$ be a proper morphism between varieties. We say that $f$ is a \emph{fibration} if $f_*\mathcal{O}_X = \mathcal{O}_Y$ holds.
    
    \item Let $f\colon X\to Y$ be a fibration of quasi-projective varieties. We say that $f$ is \emph{birationally isotrivial} if, for any two closed points $y_1,y_2\in Y$, the varieties $X_{y_1}$ and $X_{y_2}$ are birational to each other. Furthermore, we say that $f$ is \emph{generically birationally isotrivial} if it is birationally isotrivial over a non-empty open subset of $Y$.
    
    \item A variety $X$ is said to be \emph{rationally chain connected} if any two general points on $X$ can be joined by a chain of rational curves.
\end{itemize}

\subsection{Acknowledgements}\label{subsec:acknowledgements}

The author would like to thank his advisor Zsolt Patakfalvi for numerous fruitful discussions and continuous guidance throughout this project. He would also like to thank Jefferson Baudin, Stefano Filipazzi, Nikolaos Tsakanikas and Domenico Valloni for useful discussions and/or comments on the content of this article. The author was partly supported by the  project grant \#200020B/192035 from the Swiss National Science
Foundation (FNS), as well as by the ERC starting grant \#804334.

\section{Preliminaries}\label{sec:Albanese_preliminaries}

\subsection{Torus actions}\label{subsec:torus_actions}
In this section we gather some elementary properties of torus actions on algebraic varieties, most importantly on their fixed point loci. In \autoref{subsubsec:MMP_torus} we will discuss torus actions on pairs. We refer the reader to \cite[Chapter III,\S 8]{Borel_Linear_algebraic_groups} for an introduction to tori. For the convenience of the reader, we included detailed proofs for most of these elementary properties. Let us start with some definitions and terminology.

\begin{definition}\label{def:torus_action}
    An action of a torus $\bT$ on a separated scheme $X$ of finite type over $\bC$ is a morphism $\alpha\colon\bT\times_{\bC} X\to X$ satisfying the usual commutative diagrams. We write this as $\bT\acts X$. The kernel $\ker(\bT\acts X)$ is defined as the algebraic subgroup of elements of $\bT$ that act trivially on $X$ (which is a subgroup scheme by \cite[II.1.3.6]{Demazure_Gabriel_Groupes_algebriques}). If $\ker(\bT\acts X)$ is trivial, the action $\bT\acts X$ is said to be \emph{effective}.
\end{definition}

\begin{remark}\label{rem:effective_action}
    We can always reduce to the case of an effective torus action upon replacing $\bT$ with $\bT/\ker(\bT\acts X)$. This is a torus by \cite[Proposition 8.7 and Corollary on page 114, Section III.8.5]{Borel_Linear_algebraic_groups}.
\end{remark}

Now we define the complexity of the torus action $\bT\acts X$. For an effective action, it is just defined as the number $\dim X-\dim\bT$, but it gives us some more flexibility if for some reason we do not want to pass to an effective action using \autoref{rem:effective_action}.

\begin{definition}\label{def:complexity}
    Let $\bT\acts X$ be a torus action on a variety $X$. The \emph{complexity} $c(\bT\acts X)$ is defined by the formula
    \begin{align*}
        c(\bT\acts X)\coloneqq\dim X-\dim\bT+\dim\ker(\bT\acts X).
    \end{align*}
    In particular, if $\bT\acts X$ is effective (or, more generally, has finite kernel), we have $c(\bT\acts X)=\dim X-\dim\bT$. 
\end{definition}

Geometrically, the complexity is the codimension of a general orbit, and if it is $0$, then $X$ is toric. Although the above definition is sufficient for us, we include a proof for the sake of completeness.

\begin{lemma}\label{lem:complexity}
    For a torus action $\bT\acts X$ on a variety $X$, we have
    \begin{align*}
        c(\bT\acts X)=\codim\overline{\bT\cdot x}
    \end{align*}
    for a general closed point $x\in X$. In particular, if $c(\bT\acts X)=0$, then $X$ is toric.
\end{lemma}

\begin{proof}
    By \cite[Corollary 1.6]{Martin_Generic_stabilisers_for_actions_of_reductive_groups}, there exists an algebraic subgroup $\bK\subseteq\bT$ such that a general point of $x$ has stabilizer $\bK$ (because every subgroup of $\bT$ is reductive and $\bT$ is commutative). In particular, $\bK$ fixes a general point of $X$, so $\bK\acts X$ is trivial, from which we conclude that $\bK=\ker(\bT\acts X)$. By the Orbit-stabilizer theorem, we are done. Finally, note that if $c(\bT\acts X)=0$ and $x$ is general, then $\bT\cdot x$ is dense. As it is constructible by Chevalley's theorem, it contains a non-empty open subset. Through the action of $\bT$, we see that $\bT\cdot x$ is open itself.
\end{proof}

Given a torus action $\alpha\colon\bT\times_{\bC} X\to X$, a closed subscheme $Y$ of $X$ is said to be $\bT$--stable if $\alpha|_{\bT\times_{\bC}Y}$ factors through $Y$, so that we obtain an induced action $\bT\acts Y$. The action $\bT\acts X$ is said to be trivial if it is the projection to $X$, and a $\bT$--stable subscheme $Y$ is said to be $\bT$--invariant if the induced action is trivial. An important notion for this article is the fixed point scheme of an action, which was constructed by Fogarty in \cite{Fogarty_Fixed_point_schemes}.

\begin{definition}[\protect{\cite[Theorem 2.3]{Fogarty_Fixed_point_schemes}}]
    For a torus action $\bT\acts X$, the fixed point locus $X^{\bT}$ is defined to be the maximal $\bT$--invariant closed subscheme of $X$.
\end{definition}

Note that by \cite[Theorem 5.2]{Fogarty_Fixed_point_schemes}, if $X$ is smooth, then $X^{\bT}$ is as well. On the other hand, $X^{\bT}$ might be badly behaved in general; see, e.g., \cite[6.3]{Fogarty_Fixed_point_schemes} for an example where $X$ is factorial and $X^{\bT}$ is non--reduced.

Now let us collect some useful facts about torus actions. The following lemma shows that as tori are connected, a torus action cannot permute the irreducible components non--trivially.

\begin{lemma}\label{lem:action_irred_comp}
    Let $X$ be a separated scheme of finite type over $\bC$ with a torus action $\bT\acts X$. Then every irreducible component of $X^{\red}$ is $\bT$--stable.
\end{lemma}

\begin{proof}
    As $\bT$ is reduced, the product $\bT\times X^{\red}$ is reduced as well, and thus the composition
    \[
        \bT \times X^{\red} \to \bT \times X \to X
    \]
    factors through $X^{\red}$. In other words, $X^{\red}$ is $\bT$--stable. Denote by $\alpha\colon \bT\times X^{\red}\to X^{\red}$ the action, and let $X'$ be an irreducible component of $X^{\red}$. As $\bT$ is irreducible by \stacksproj{0B7Q}, we obtain that $\bT\times X'$ is irreducible. Hence the scheme theoretic image of $\alpha|_{\bT\times X'}$ is irreducible and contains $X'$, and so is in fact equal to $X'$. In particular, $\alpha|_{\bT\times X'}$ factors through $X'$, i.e. $X'$ is $\bT$--stable.
\end{proof}

The following lemma is to prepare \autoref{lem:dim_T_action}. As the dimension is topological in nature, we want to be able to pass to the reduction, to avoid dealing with non--reducedness.

\begin{lemma}\label{lem:fixed_red}
    Let $X$ be a separated scheme of finite type over $\bC$ with a torus action $\bT\acts X$. Then the reduction $X^{\red}\to X$ induces a canonical isomorphism
    \begin{align*}
        ((X^{\red})^{\bT})^{\red}=(X^{\bT})^{\red}.
    \end{align*}
\end{lemma}

\begin{proof}
    This can be seen on the level of functors. By \cite{Fogarty_Fixed_point_schemes}, $X^{\bT}$ represents the functor
    \begin{align*}
        h_{X^\bT} \colon \Sch_{\bC} &\longrightarrow \Set\\
        Y &\longmapsto \GSch_{\bC}(Y_\bT,X),
    \end{align*}
    where $Y_\bT$ denotes $Y$ equipped with the trivial $\bT$--action. Therefore, $(X^\bT)^{\red}$ represents the restricted functor $h_{X^\bT}|_{\Sch_{\bC}^{\red}}$, where $\Sch_{\bC}^{\red}$ denotes the full subcategory of reduced schemes. Hence it suffices to see that we have a natural isomorphism of functors
    \begin{align*}
        h_{(X^{\red})^\bT}|_{\Sch_{\bC}^{\red}}\cong h_{X^\bT}|_{\Sch_{\bC}^{\red}}.
    \end{align*}
    This is straightforward to see: let $Y\in \Sch_{\bC}^{\red}$ be arbitrary, then
    \begin{align*}
        %\hskip0.147\textwidth
        %h_{(X^{\red})^\bT}(Y)=\GSch_{\bC}(Y_\bT,X^{\red})\cong \GSch_{\bC}(Y_\bT,X)=h_{X^\bT}(Y).\hskip0.147\textwidth\qedhere
        \fillqed{h_{(X^{\red})^\bT}(Y)=\GSch_{\bC}(Y_\bT,X^{\red})\cong \GSch_{\bC}(Y_\bT,X)=h_{X^\bT}(Y).}
    \end{align*}
\end{proof}

The following estimate for the dimension of the fixed locus of a torus action is key to the argument. It is a generalization of the fact that on a toric variety, there are at most finitely many points fixed by the torus action. Most likely it is known in the literature, but as we were unable to find an exact reference, and as parts of the proof will be useful later on, we include an elementary proof.

\begin{lemma}\label{lem:dim_T_action}
    Let $X$ be a variety with a torus action $\bT\acts X$. Then
    \begin{align*}
        \dim X^{\bT} \leq c(\bT \acts X).
    \end{align*}
\end{lemma}

\begin{proof}
    We proceed by induction on the dimension of $X$: the case $\dim X=0$ is trivial. So suppose that the statement is established for varieties of dimension strictly smaller than $\dim X>0$. By replacing $\bT$ with the quotient of $\bT$ by the kernel of $\bT\acts X$, we may assume that $\bT\acts X$ is effective. If $X^{\bT}$ is finite (or empty) then there is nothing to prove. So denote $d\coloneqq \dim X^{\bT}$ and assume $d>0$ without loss of generality. Let $z\in X^{\bT}$ be a closed point with $\dim_z X^{\bT}=d$. Consider the tangent cone $C(X,z)$ of $X$ at $z$ as a closed subscheme of the Xariski tangent space $T_z(X)$. As $z$ is a fixed point, we obtain a linear action of $\bT$ on $T_z(X)$ for which $C(X,z)$ is $\bT$-stable. We now proceed through a series of claims.\medskip

    \noindent\textbf{Claim 1:} \textit{There exists an irreducible component $C$ of $C(X,z)^{\red}$ such that $\dim C=\dim X$ and $\dim C^{\bT}=d$.}\medskip
    
    \noindent\textit{Proof of Claim 1:} Let $C(X,z)^{\red}=C_1\cup\cdots\cup C_r$ be the decomposition into irreducible components. Note that by \autoref{lem:action_irred_comp}, every $C_i$ is stable under the action of $\bT$. Therefore, by \cite[Theorem 5.2]{Fogarty_Fixed_point_schemes} we have
    \begin{align*}
        C(X^{\bT},z)^{\red}\expl{=}{\cite[Theorem 5.2]{Fogarty_Fixed_point_schemes}}(C(X,z)^{\bT})^{\red}\expl{=}{\autoref{lem:fixed_red}}((C(X,z)^{\red})^{\bT})^{\red}\expl{=}{\autoref{lem:action_irred_comp}}(C_1^{\bT})^{\red}\cup\cdots\cup(C_r^{\bT})^{\red}.
    \end{align*}
    As $\dim C(X^{\bT},z)^{\red}=d$, there must exist an $i$ such that $C_i^{\bT}$ has dimension $d$. Finally, as $X$ is irreducible, $C(X,z)$ is equidimensional of dimension $\dim X$, so in particular $\dim C=\dim X$.\qedclaim\medskip

    \noindent\textbf{Claim 2:} \textit{We have $c(\bT\acts C)=c(\bT\acts X)$.}\medskip
    
    \noindent\textit{Proof of Claim 2:}
    Let $\bK$ be the kernel of $\bT\acts C$, and let $\bK_0\leq \bK$ be the connected component of the identity. Then $\bK_0$ is a torus, and thus 
    \begin{align*}
        C(X^{\bK_0},z)^{\red}\expl{=}{\cite[Theorem 5.2]{Fogarty_Fixed_point_schemes}}(C(X,z)^{\bK_0})^{\red}\expl{=}{\autoref{lem:fixed_red}}((C(X,z)^{\red})^{\bK_0})^{\red}\supseteq (C^{\bK_0})^{\red}=C.
    \end{align*}
    Therefore, we obtain 
    \begin{align*}
        \dim_z X^{\bK_0}=\dim C(X^{\bK_0},z)\geq\dim C= \dim X.
    \end{align*}
    Hence we must have $X^{\bK_0}=X$. As $\bT\acts X$ is effective, $\bK_0$ must be trivial, so $\bK$ is a finite subgroup of $\bT$. Hence we obtain
    \begin{align*}
        c(\bT\acts C) = \dim X - \dim \bT +\underbrace{\dim\bK}_{=0}= c(\bT \acts X),
    \end{align*}
    so we are done.\qedclaim\medskip
    
    Now as $\bT\acts T_z(X)$ is linear and $C\subseteq T_z(X)$ is a cone, we obtain an action of $\bT$ on the projectivization $\bP C$ of $C$.\medskip

    \noindent\textbf{Claim 3:} \textit{We have $c(\bT\acts \bP C)=c(\bT\acts X)-1$.}\medskip
    
    \noindent\textit{Proof of Claim 3:} As $\dim\bP C=\dim X-1$, we have
    \begin{align*}
        c(\bT\acts \bP C)\geq \dim (\bP C)-\dim\bT=\dim X-1-\dim\bT=c(\bT\acts X)-1.
    \end{align*}
    Suppose by contradiction that the inequality is strict; then the kernel of $\bT\acts\bP C$ must be infinite. As the kernel of $\bT\acts C$ is finite by Claim 2, there must exist $t\in \bT$ acting non-trivially on $C$ but trivially on $\bP C$. So every point of $C$ is an eigenvector of $t$, or in other words, $C$ is contained in the union of eigenspaces of $t$. As they are finite in number and $C$ is irreducible, $C$ must be contained in a single eigenspace of $t$, i.e., there exists $\lambda\in\bC$ such that $t.v=\lambda v$ for all $v\in C$. As $t$ acts non-trivially on $C$ we must have $\lambda\neq 1$. But then $C^{\bT}=\{0\}$, which by Claim 1 gives $d=0$, contradiction.
    \qedclaim\medskip

    We can now conclude the induction step: as $\dim \bP C=\dim X-1$, we have by the induction hypothesis
    \begin{align*}
        c(\bT\acts X)-1\expl{=}{\text{Claim 3}} c(\bT\acts\bP C)\geq \dim(\bP C)^{\bT}\geq \dim \bP(C^{\bT})=\dim C^{\bT}-1\expl{=}{\text{Claim 1}}d-1.
    \end{align*}
    Adding $1$ to both sides then concludes the proof.
\end{proof}

The following lemma discusses the equality case of \autoref{lem:dim_T_action} for linear torus actions on a vector space. We will need it to extend such torus actions to a partial compactification of $\bT$, see \autoref{subsec:Birula}.

\begin{lemma}\label{lem:extend_lin_t_action}
    Let $V$ be a finite dimensional vector space and let $\bT\acts V$ be a linear torus action. Assume that we have equality for the bound given by \autoref{lem:dim_T_action}, i.e.
    \begin{align*}
        \dim V^{\bT}=c(\bT\acts V).
    \end{align*}
    Then there exists an isomorphism $\bT\cong\bG_m^{e}$ such that the corresponding weight decomposition of $V$ has no negative weights.
\end{lemma}

\begin{proof}
    Denote the connected component of the identity of the kernel of $\bT\acts V$ by $\bK$ and $\bT'\coloneqq \bT/\bK$. As short exact sequences of tori are split (see \cite[Corollary on page 114, Section III.8.5]{Borel_Linear_algebraic_groups}) and $\bK\acts V$ is trivial, it suffices to prove the statement for $\bT'$. Hence we may replace $\bT$ by $\bT'$, so that without loss of generality, we may assume that $\bT\acts V$ has finite kernel. In particular, we have $c(\bT\acts V)=\dim V-\dim \bT$.
    
    Denote $n=\dim V$, $d=\dim V^{\bT}$ and $e=\dim \bT$, so that by hypothesis $d+e=n$. Fix some isomorphism $\Phi\colon \bT\tilde{\to}\bG_m^e$, and for $t\in \bT$ we denote 
    \[
        \Phi(t):=\underline{t}=(t_1,\ldots,t_e).
    \]
    Consider the corresponding weight decomposition of $\bG_m^e\acts V$ (see e.g. \cite[Chapter III, \S8]{Borel_Linear_algebraic_groups}): there exists a basis $v_1,\ldots,v_n$ of $V$ such that for all $t\in \bT$ and $1\leq j\leq n$ we have 
    \[
        t\cdot v_j=\underline{t}^{m_{\bullet j}}v_j,
    \] 
    where $m_{\bullet j}=(m_{1j}\ \ldots \ m_{ej})^{\intercal}\in\bZ^{e}$ and 
    \[
        \underline{t}^{m_{\bullet j}}\coloneqq t_1^{m_{1j}}\cdots t_e^{m_{ej}}.
    \]
    In particular, the dimension $d$ of $V^{\bT}$ is equal to the number of $0$--columns of the integer matrix $M=(m_{ij})_{\substack{1\leq i\leq e\\ 1\leq j\leq n}}$. As $d+e=n$, we may assume, up to reordering the $v_j$'s, that the $0$--columns of $M$ are precisely $m_{\bullet,e+1},\ldots,m_{\bullet n}$. Denote $S=\operatorname{Span}\{v_1,\ldots,v_e\}$ and $M'=(m_{ij})_{1\leq i,j\leq e}$, so that the action of $\bT$ on $S$ is given by $M'$. Note that as $\bT\acts V$ has finite kernel and $V=S\oplus V^{\bT}$, the action $\bT\acts S$ also has a finite kernel. In particular, $M'$ has non-zero determinant, i.e. is invertible over $\bQ$.

    We now claim that there exists $L\in \GL_e(\bZ)$ such that $LM'$ has only non-negative entries. Indeed, we can triangularize $M'$ by repeated application of Euclid's Algorithm: there exists $L'_1\in \GL_e(\bZ)$ such that $(L'_1m_{\bullet1})^{\intercal}=(*\ 0\ \ldots\ 0)$, and then we repeat the procedure with $((L'_1M')_{i,j})_{2\leq i,j\leq e}$. At the end, we obtain $L'\in\GL_e(\bZ)$ such that $L'M'$ is upper-triangular. As $\det(M')\neq 0$, the elements on the diagonal are non-zero, and hence up to multiplying some rows of $L'$ by $-1$, we may suppose that they are strictly positive. But then we can use these elements on the diagonal to turn every element above positive with appropriate row operations. In the end, we obtain $L\in\GL_e(\bZ)$ such that $LM'$ is upper-triangular with non-negative entries.

    Now denote by $\varphi$ the automorphism of $\bG_m^e$ given by
    \begin{align*}
        \varphi\colon\bG_m^e &\to \bG_m^e\\
        \underline{\lambda}&\mapsto (\underline{\lambda}^{L_{\bullet1}},\ldots,\underline{\lambda}^{L_{\bullet e}}),
    \end{align*}
    and let $\Psi=\varphi^{-1}\circ\Phi$ be an adjusted identification $\bT\cong\bG_m^{e}$. Under this identification, the weight decomposition of $\bT\acts V$ is given by the matrix $N\coloneqq (LM'\ 0_{e\times d})$, so in particular it has only non-negative entries.
\end{proof}

\subsection{Minimal model program}\label{subsec:Albanese_MMP}

We recall the definitions of the MMP singularities of a pair $(X,\Delta)$, see \cite{Kollar_Singularities_of_the_minimal_model_program}. Let $\pi \colon Y \to X$ be a proper birational morphism from a normal variety $Y$. We choose a canonical divisor $K_Y$ such that $\pi_*K_Y=K_X$, and we write
\begin{align*}
    K_Y + \pi_*^{-1}\Delta = \pi^*(K_X+\Delta)+\sum_i a(E_i, X, \Delta) E_i,
\end{align*}
where $\{E_i\}_i$ is the set of prime exceptional divisors and $\pi_*^{-1}\Delta $ is the strict transform of $\Delta$. 
The coefficient $a(E, X, \Delta)$ is called the \emph{discrepancy} of $(X, \Delta)$ along $E$, and it depends
exclusively on the divisorial valutation associated to $E$.

\begin{definition}\label{def:klt_lc}
    A pair $(X, \Delta)$ is said to be \emph{log canonical} (resp.\ \emph{Kawamata log terminal}) if for every $E$ as above, $a(E, X, \Delta) \geq -1$ (resp.\ $a(E, X, \Delta) > -1$ and $\lfloor \Delta \rfloor = 0$). We usually abbreviate `log canonical' to `lc' and `Kawamata log terminal' to `klt'.
\end{definition}

A useful class of singularities sitting in between klt and lc are \emph{divisorial log terminal} (\emph{`dlt'} for short) singularities. Given an lc pair $(X, \Delta)$, an irreducible subvariety $Z$ is called an \emph{lc center} if there exists a divisor $E$ on a birational model of $X$ of discrepancy $-1$ and whose center on $X$ is equal to $Z$. A \emph{minimal lc center} is an lc center that is minimal for inclusion among lc centers.

\begin{definition}\label{def:dlt}
    An lc pair $(X, \Delta)$ is said to be \emph{dlt} if $(X, \Delta)$ has simple normal crossing (see \cite[Definition 1.7]{Kollar_Singularities_of_the_minimal_model_program}) at the generic point of every lc center.
\end{definition}

% \subsubsection{MMP and torus actions}\label{subsubsec:MMP_torus}

% We now state some useful facts about the MMP when torus actions are involved. Note that everything would work the same for the action of an arbitrary connected algebraic group. For the definitions of lc, dlt, resp.\ klt singularities, and of (minimal) lc centers, we refer the reader to \autoref{def:klt_lc}, \autoref{def:dlt} and the paragraph in between. We start with the following remark.

\begin{remark}\label{rem:dlt_smooth_locus}
    In particular, if $Z$ is an lc center in a dlt pair $(X,\Delta)$, then $X$ is smooth at the generic point of $Z$. This implies, e.g., that $X^{\sm}\cap Z\neq \emptyset$.
\end{remark}

Next, we define the notions `$K$--trivial' and `log Calabi--Yau pair'.

\begin{definition}\label{def:K-trivial}
    An lc pair $(X,\Delta)$ is said to be $K$--trivial if $K_X+\Delta\equiv 0$. If $(X,\Delta)$ is projective, then by \cite[Theorem 1.2]{Gongyo_Abundance_theorem_for_numerically_trivial_log_canonical}, this is equivalent to $K_X+\Delta\sim_{\bQ}0$. We then call $(X,\Delta)$ a log Calabi--Yau pair.
\end{definition}

\subsubsection{MMP and torus actions}\label{subsubsec:MMP_torus}

We now state some useful facts about the above notions when torus actions are involved. Note that everything would work the same for the action of an arbitrary connected algebraic group. We start by defining torus actions on pairs.

\begin{definition}
    Let $(X,\Delta)$ be a pair and $\bT$ a torus. An action of $\bT$ on $(X,\Delta)$ is an action $\bT\acts X$ such that $t_{*}\Delta=\Delta$ for all $t\in\bT$. Note that by \autoref{lem:action_irred_comp}, this is equivalent to requiring that every component of $\Supp(\Delta)$ is $\bT$--stable. We write $\bT\acts (X,\Delta)$ to denote such an action.
\end{definition}

We now discuss \emph{dlt modifications}, which provide a useful technique to pick a birational model with favorable properties. In the next lemma, we show that this construction can be done $\bT$--equivariantly. Essentially, this is possible because singularities can be resovled functorially, so that in particular any group action lifts to the resolution, and running the MMP is a discrete procedure, so it has to be preserved under the action of a connected group.

\begin{lemma}[\protect{\cite[Corollary 1.36]{Kollar_Singularities_of_the_minimal_model_program}}]\label{lem:dlt_mod}
    Let $(X,\Delta)$ be an lc pair and suppose we have a torus action $\bT\acts (X,\Delta)$. Then there exists a dlt pair $(X^{\dlt},\Delta^{\dlt})$ and a proper birational morphism $g\colon X^{\dlt}\to X$ such that
    \begin{enumerate}[start=0]
        \item\label{lem:dlt_mod_T-equivariant} there exists $\bT\acts(X^{\dlt},\Delta^{\dlt})$ such that $g$ is $\bT$--equivariant and $c(\bT\acts X^{\dlt})=c(\bT\acts X)$;
        \item\label{lem:dlt_mod_factorial} $X^{\dlt}$ is $\bQ$--factorial;
        \item\label{lem:dlt_mod_crepant}$K_{X^{\dlt}}+\Delta^{\dlt}\sim_{\bQ} g^{*}(K_X+\Delta)$.
    \end{enumerate}
\end{lemma}

\begin{proof}
    We follow the proof of \cite[Corollary 1.36]{Kollar_Singularities_of_the_minimal_model_program}, keeping track of the torus action at every step. We start with a functorial log resolution $f\colon(\tX,\widetilde{\Delta})\to(X,\Delta)$, which exists by \cite[Theorem 4.4]{Greb_Kebekus_Kovacs_Extension_theorems_for_differential_forms_and_Bogomolov_Sommese_vanishing_on_log_canonical_varieties}, and where $\widetilde{\Delta}$ is the total transform of $\Delta$ (see \cite[Definition 2.4]{Greb_Kebekus_Kovacs_Extension_theorems_for_differential_forms_and_Bogomolov_Sommese_vanishing_on_log_canonical_varieties}). By functoriality, the $\bT$--action on $(X,\Delta)$ lifts to a $\bT$--action on $(\tX,\widetilde{\Delta})$ such that $f$ is $\bT$--equivariant, see \cite[Proposition 3.9.1]{Kollar_Lectures_on_resolution_of_singularities}. Now to construct a dlt--model $(X^{\dlt},\Delta^{\dlt})$ of $(X,\Delta)$, we run a $(K_{\wt{X}}+f_*^{-1}\Delta+\sum_i E_i)$--MMP on $\tX$ where $\{E_i\}_i$ is the set of $f$--exceptional divisors, see \cite[1.35]{Kollar_Singularities_of_the_minimal_model_program}. Note that $\bT$ still acts on $(\tX, f_*^{-1}\Delta+\sum_i E_i)$: as the total transform $\widetilde{\Delta}$ is $\bT$--stable also the strict transform $f_*^{-1}\Delta$ is (again by \autoref{lem:action_irred_comp}), and also the $E_i$ are $\bT$--stable by functoriality of the resolution. This results in a dlt model $g\colon (X^{\dlt},\Delta^{\dlt})\to (X,\Delta)$, which by \cite[Corollary 1.36]{Kollar_Singularities_of_the_minimal_model_program} satisfies \autoref{lem:dlt_mod_factorial} and \autoref{lem:dlt_mod_crepant}. Then, by \cite[Remark 2.3]{Floris_A_note_on_the_G-Sarkisov_program} (which is based on Blanchard's lemma, see, e.g., \cite[Proposition 4.2.1]{Brion_Samuel_Uma_Lectures_on_the_structure_of_algebraic_groups_and_geometric_applications}), this MMP is automatically a $\bT$--MMP, so that in particular we obtain a $\bT$--action on $(X^{\dlt},\Delta^{\dlt})$. Note that the kernel of $\bT\acts X^{\dlt}$ agrees with the kernel of $\bT\acts X$, as $g\colon X^{\dlt}\to X$ is surjective. In particular, we have $c(\bT\acts X^{\dlt})=c(\bT\acts X)$, so we are done. 
\end{proof}

We now move to discussing lc centers. On the one hand, an important property of lc centers is that despite potentially having higher codimension, we have a version of adjunction for them. On the other hand, there are only finitely many of them, so again they have to be preserved under torus actions. This is the content of the following lemma.

\begin{lemma}[\protect{\cite[4.29, Theorem 4.19]{Kollar_Singularities_of_the_minimal_model_program}}]\label{lem:min_lc_fixed}
    Let $(X,\Delta)$ be a projective lc pair with a torus action $\bT\acts (X,\Delta)$, and let $Z$ be an lc center. Then
    \begin{enumerate}
        \item\label{lem:min_lc_fixed_stable} $Z$ is $\bT$--stable;
        \item\label{lem:min_lc_fixed_adjunction} if $X$ is dlt, then $Z$ is normal and there exists a boundary $\Diff^*_Z\Delta$ on $Z$ such that
        \begin{enumerate}
            \item\label{lem:min_lc_fixed_adjunction_T-equivariant} $\bT$ acts on $(Z,\Diff^*_Z\Delta)$;
            \item\label{lem:min_lc_fixed_adjunction_singularities} $(Z,\Diff^*_Z\Delta)$ is dlt, and if $Z$ is minimal, this pair is klt;
            \item\label{lem:min_lc_fixed_adjunction_formula} we have $(K_X+\Delta)|_Z\sim_{\bQ} K_Z+\Diff^*_Z\Delta$;
        \end{enumerate}
        \item\label{lem:min_lc_fixed_invariant} if $Z$ is minimal and $(X,\Delta)$ is $K$--trivial, then 
        $Z$ is $\bT$--invariant.
    \end{enumerate}
\end{lemma}

\begin{proof}
    Let $\sZ$ be the set of lc centers of $(X,\Delta)$ of dimension $d\coloneqq\dim Z$. Note that $\sZ$ is finite by \cite[Theorem 1.1]{Ambro_Basic_properties_of_log_canonical_centers}. Furthermore, for every $t\in\bT$ and every $Z'\in\sZ$, as $t$ is an automorphism of $X$, we have that $t(Z')$ is an lc center of dimension $d$ of $(X,t_*\Delta)=(X,\Delta)$. Hence the closed subset $\bigcup_{Z'\in\sZ}Z'$ is $\bT$--invariant. By \autoref{lem:action_irred_comp}, all the components are $\bT$-invariant, which proves \autoref{lem:min_lc_fixed_stable}.

    Now assume that $X$ is dlt, then $Z$ is normal by \cite[Theorem 4.16]{Kollar_Singularities_of_the_minimal_model_program}. The construction of $\Diff^*_Z\Delta$ is done via higher Poincaré residue maps, see \cite[4.18]{Kollar_Singularities_of_the_minimal_model_program}. Points \autoref{lem:min_lc_fixed_adjunction_singularities} and \autoref{lem:min_lc_fixed_adjunction_formula} are given by \cite[Theorem 4.19]{Kollar_Singularities_of_the_minimal_model_program} (if $Z$ is minimal, then by Theorem 4.19.3 the pair $(Z,\Diff^*_Z\Delta)$ has no non-trivial lc centers and thus is klt). So we are left to verify that the construction of $\Diff^*_Z\Delta$ is compatible with the action of $\bT$. If $\sigma$ is an automorphism of $X$, then from the construction in \cite[4.18]{Kollar_Singularities_of_the_minimal_model_program} it is straightforward to see that $\sigma_{*}\Diff^*_Z\Delta=\Diff^*_{\sigma(Z)}\sigma_*\Delta$. Therefore, as $\bT$ acts on $(X,\Delta)$, we obtain that $\bT$ acts on $(Z,\Diff^*_Z\Delta)$, as $t(Z)=Z$ for all $t\in\bT$ by \autoref{lem:min_lc_fixed_stable}.

    To prove \autoref{lem:min_lc_fixed_invariant}, we drop the dlt assumption, but instead assume that $(X,\Delta)$ is lc, $K$--trivial and that $Z$ is minimal. Let $g\colon(X^{\dlt},\Delta^{\dlt})\to(X,\Delta)$ be the $\bT$--equivariant dlt model provided by \autoref{lem:dlt_mod}. Let $S$ be an lc center of $(X^{\dlt},\Delta^{\dlt})$ that is minimal for inclusion among the lc centers dominating $Z$ (see \cite[Definition 4.15]{Kollar_Singularities_of_the_minimal_model_program} and the discussion thereafter). Then $S$ is a minimal lc center, because if by contradiction there exists an lc center $S'$ of $(X^{\dlt},\Delta^{\dlt})$ with $S'\subsetneq S$, then $g(S')$ is an lc center of $(X,\Delta)$ with $g(S')\subsetneq Z$, contradicting the minimality of $Z$. To ease the notation, denote $\Delta_S=\Diff^*_{S}\Delta^{\dlt}$. Then by \autoref{lem:min_lc_fixed_adjunction}, $(S,\Delta_S)$ is klt, $K$--trivial and $\bT\acts(X^{\dlt},\Delta^{\dlt})$ induces an action $\bT\acts (S,\Delta_S)$.
    
    Let us show that $\bT\acts S$ is trivial. By \cite[Theorem 1]{Sumihiro_Equivariant_completion}, there exists an embedding $S\hookrightarrow\bP^n$ and a linear action $\bT\acts\bP^n$ inducing the action $\bT\acts S$. In other words, there exists a polarization $\cL$ on $S$ such that $t^*\cL\cong\cL$ for all $t\in \bT$. But on the other hand, by \cite[Proposition 10.1]{Patakfalvi_Zdanowicz_On_the_Beauville-Bogomolov_decomposition_in_characteristic_p}, the abstract group
    \begin{align*}
        \Aut(S,\Delta_S;\cL)\coloneqq \{\sigma\in\Aut(S)(\bC)\mid \sigma_*\Delta_S=\Delta_S,\ \sigma^*\cL\cong\cL\}
    \end{align*}
    is finite. Therefore, $\bT/\ker(\bT\acts S)$ has to be finite, as its closed points inject into a finite set. As it is a torus, it must thus be trivial, i.e., $\bT=\ker(\bT\acts S)$. Finally, as $g\colon X^{\dlt}\to X$ is $\bT$--equivariant, the (co)restriction $g\colon S\to Z$ is $\bT$--equivariant as well. As it is also surjective, we obtain that $\bT\acts Z$ is trivial, and thereby conclude the proof.
\end{proof}

In particular, we obtain that minimal lc centers are in the fixed locus of any torus action. Together with \autoref{lem:dim_T_action}, this gives the following.

\begin{corollary}\label{cor:dim_bound_torus_center}
    Let $(X,\Delta)$ be a log Calabi--Yau pair, let $Z$ be a minimal lc center and $\bT\acts(X,\Delta)$ a torus action. Then we have $\dim Z\leq c(\bT\acts X)$, and in the case of equality, $Z$ is a top--dimensional component of $X^{\bT}$. In particular, if a torus $\bT\subseteq\Aut(X)$ preserves $\Delta$, then $\dim\bT\leq\codim Z$.
\end{corollary}

\subsection{Albanese morphism of open varieties}\label{subsec:Albanese}

As the conclusion of \autoref{thmA} is on the one hand about the Albanese morphism and on the other hand of a birational nature, it will be useful to have a notion of Albanese morphism also for open varieties. A modern and much more refined and general construction can be found in \cite{Fujino_On_quasi-Albanese_maps, Schroer_Albanese_maps_for_open_algebraic_spaces}, but for us the construction in \cite{Serre_Morphismes_universels_et_variete_dAlbanese} will be sufficient.

\begin{definition}[\protect{\cite[Theorem 5]{Serre_Morphismes_universels_et_variete_dAlbanese}}]\label{def:Albanese}
    Let $U$ be a variety. Then there exists an abelian variety $\Alb_U$ and a morphism $\alb_U\colon U\to \Alb_U$ which is the universal map to an abelian variety from $U$. By virtue of being universal, for a morphism $f\colon U\to W$, there exists a unique morphism $\Alb(f)\colon \Alb_U\to\Alb_W$ such that $\alb_W\circ f=\Alb(f)\circ\alb_U$. In other words, $\Alb$ is a functor.
\end{definition}

\begin{remark}\label{rem:Albanese}
    Note that by \cite[Remark 9.5.25]{Fantechi_Goettsche_Illusie_Kleiman_Nitsure_Vistoli_Fundamental_algebraic_geometry}, if $U$ is normal and projective, then $\Alb_U\cong(\Pic^0_U)_{\red}^{\vee}$. 
\end{remark}

To compare the Albanese of an open variety with the Albanese of some compactification, we will need the following lemma, which shows that the Albanese of a variety can be computed as the Albanese of a big enough compactification. It is a vastly simplified version of the construction in \cite[Section 5]{Schroer_Albanese_maps_for_open_algebraic_spaces}.

\begin{lemma}\label{lem:alb_of_big_comp}
    Let $U$ be a variety and let $X$ be a compactification of $U$. Then there exists a compactification $Y$ of $U$ dominating $X$, such that for any compactification $Z$ of $U$ dominating $Y$, the inclusion $U\hookrightarrow Z$ induces an isomorphism $\Alb_U\cong\Alb_Z$.
\end{lemma}

\begin{proof}
    We have a rational map $X\dashrightarrow \Alb_U$, so we can resolve the indeterminacy to obtain a compactification $Y$ of $U$, dominating $X$, and a map $Y\to\Alb_U$. Hence, if $Z$ is any compactification of $U$ dominating $Y$, we obtain maps $\Alb_Z\to \Alb_U$ and $\Alb_U\to \Alb_Z$, and by the universal property of the Albanese, they must be mutually inverse.
\end{proof}

As the next lemma shows, for smooth varieties, taking a smooth compactification is already sufficient, essentially because the Albanese is a birational invariant of smooth projective varieties. We will later improve this to the setting of dlt singularities in \autoref{lem:alb_iso_dlt}.

\begin{lemma}\label{lem:alb_iso_smooth}
    Let $X$ be a smooth projective variety and let $i\colon U\hookrightarrow X$ be an open immersion. Then $\Alb(i)\colon\Alb_U\to\Alb_X$ is an isomorphism.
\end{lemma}

\begin{proof}
    By \autoref{lem:alb_of_big_comp} there exists a compactification $j\colon U\to Y$ of $U$ such that $\Alb(j)$ is an isomorphism and $Y$ dominates $X$ via some map $f\colon Y\to X$. Furthermore, by passing to a resolution of singularities of $Y$ which is an isomorphism over $U$, we may assume that $Y$ is smooth. As the Albanese is a birational invariant for smooth projective varieties, $\Alb(f)$ is an isomorphism. Hence $\Alb(i)=\Alb(f)\circ\Alb(j)$ is an isomorphism as well.
\end{proof}

An elementary but useful property of the Albanese variety is that it is trivial for rationally chain connected varieties, as there are no rational curves on abelian varieties. This in turn implies that a morphism whose fibers are rationally chain connected induces an isomorphism of Albanese varieties, see \cite[Lemma 2.10]{Bernasconi_Filipazzi_Patakfalvi_Tsakanikas_A_counterexample_to_the_log_canonical_Beauville-Bogomolov_decomposition}. For our purposes, we need a version of this for open varieties. As by \autoref{lem:alb_iso_smooth} the Albanese of open subsets of $\bP^n$ is trivial, we will use the following notion of rational chain connectedness of open varieties.

\begin{definition}\label{def:quasi-rcc}
    A variety $U$ is said to be quasi-rcc if for any two general points $u,u'$, there exists a sequence of non-constant morphisms $C_1\to U,\ldots,C_N\to U$ such that $C_i$ is an open subset of $\bP^1$, $u$ is in the image of $C_1$, $u'$ is in the image of $C_N$ and the images of $C_i$ and $C_{i+1}$ have non-empty intersection for all $i$.
\end{definition}

\begin{remark}\label{rem:quasi-rcc}
    The main technical difference of \autoref{def:quasi-rcc} and the classical notion of rational chain connectedness as in, e.g., \cite[Definition IV.3.2]{Kollar_Rational_curves_on_algebraic_varieties}, is that we don't require two general points to be connected by a chain of \emph{proper} rational curves, but merely by a chain of open subsets of those. Therefore, e.g., $\bA^n$ is quasi--rcc, or more generally any open subset of a rationally chain connected variety. It is therefore a rather crude notion, but as we are using also a crude notion for the Albanese morphism of an open variety, it is sufficient for us.
\end{remark}

\begin{lemma}[\protect{\cite[Lemma 2.10]{Bernasconi_Filipazzi_Patakfalvi_Tsakanikas_A_counterexample_to_the_log_canonical_Beauville-Bogomolov_decomposition}}]\label{lem:alb_iso_rcc}
    Let $f\colon U\to W$ be a surjective morphism whose fibers are quasi-rcc. If one of the following holds, then $\Alb(f)$ is an isomorphism:
    \begin{enumerate}
        \item\label{lem:alb_iso_rcc_proper} $f$ is proper and $W$ is normal;
        \item\label{lem:alb_iso_rcc_section} $f$ admits a section.
    \end{enumerate}
\end{lemma}

\begin{proof}
    By \autoref{lem:alb_iso_smooth}, the Albanese variety of any open subset of $\bP^1$ is trivial. Therefore, the Albanese of a quasi-rcc variety must be a point, as every pair of general points $u,u'$ must then have same image in the Albanese. As the fibers of $f$ are quasi-rcc by hypothesis, $\alb_U$ contracts these fibers.

    Assume first that \autoref{lem:alb_iso_rcc_proper} holds. As the fibers of $f$ are quasi-rcc, they are in particular connected, and thus $f_*\cO_U=\cO_W$ (by using the Stein factorization and Zariski's main theorem). Then the hypotheses of \autoref{lem:rigidity_magic} are satisfied, so that we obtain a morphism $h\colon W\to\Alb_U$ with $\alb_U=h\circ f$. But then $(W,h)$ satisfies the universal property of the Albanese of $W$, and thus $\Alb(f)$ must be an isomorphism.

    Assume now that \autoref{lem:alb_iso_rcc_section} holds, so let $s\colon W\to U$ be a section of $f$. We then claim that $\alb_U\circ s\circ f=\alb_U$. Indeed, it suffices to check this on closed points. If $u\in U(\bC)$ is arbitrary, then $u$ and $s(f(u))$ are in the same fiber of $f$. They are hence mapped to the same point under $\alb_U$, which proves the claim. Therefore we obtain also 
    \begin{align*}
        \Alb({\alb_U})\circ\Alb(s)\circ\Alb(f)=\Alb({\alb_U})\expl{\implies}{$\Alb({\alb_U})$ is an isomorphism} \Alb(s)\circ\Alb(f)=\id_{\Alb_U},
    \end{align*}
    by functoriality. As $s$ is a section of $f$ we also have $\Alb(f)\circ\Alb(s)=\id_{\Alb_W}$, and so $\Alb(f)$ is indeed an isomorphism.
\end{proof}

To factor the Albanese morphism in case \autoref{lem:alb_iso_rcc_proper} above, we need a version of the rigidity lemma as in \cite[Lemma 1.15(b)]{Debarre_Higher-dimensional_algebraic_geometry} without the properness assumption on the morphism contracting the fibers (which for us is $\alb_U$). Because there isn't a significant increase in complexity, we prove the following slightly more general version, inspired by \cite[Proposition 4.1]{Conrad_rigidity}.
\begin{lemma}\label{lem:rigidity_magic}
    Let $X,Y,Z$ be varieties, and let $f\colon X\to Y$ and $g\colon X\to Z$ be morphisms such that
    \begin{enumerate}
        \item\label{lem:rigidity_magic_top_quot} $f$ is a topological quotient map (e.g., it is proper and surjective)
        \item\label{lem:rigidity_magic_fibration} $f_{*}\cO_X=\cO_Z$
        \item\label{lem:rigidity_magic_contracts_fibers} for all $x,x'\in X(\bC)$ with $f(x)=f(x')$, we have $g(x)=g(x')$.
    \end{enumerate}
    Then there exists a morphism $h\colon Y\to Z$ such that $g=h\circ f$.
\end{lemma}

\begin{proof}
    Firstly, we construct $h$ topologically. To this end, we verify that $g$ is constant on the set-theoretic fibers of $f$.\medskip
    
    \noindent\textbf{Claim 1:} \textit{If $\eta,\eta'\in X$ are scheme points with $f(\eta)=f(\eta')$, then $g(\eta)=g(\eta')$.}\medskip

    \noindent\textit{Proof of Claim 1:} As irreducible subsets admit a unique generic point, it suffices to show that $\overline{\{g(\eta)\}}=\overline{\{g(\eta')\}}$. By Claim 2 and by symmetry, it suffices to show that $\overline{\{g(\eta)\}}$ contains all images under $g$ of closed points in $\overline{\{\eta'\}}$. So suppose by contradiction that there exists $x_0'\in \overline{\{\eta'\}}\cap X(\bC)$ with
    \begin{align*}
        g(x_0')\in \underbrace{Y\setminus \overline{\{g(\eta)\}}}_{V\coloneqq}.
    \end{align*}
    Then by Claim 3, $W\coloneqq f(g^{-1}(V))$ is open in $Y$. Notice that by assumption, we have $f(x_0')\in W$. Now as $f(\eta)=f(\eta')$, we obtain by Claim 2 that $f(x_0)\in W$ for some $x_0\in \overline{\{\eta\}}\cap X(\bC)$. But then as $W=f(g^{-1}(V))$, there exists a closed point $x_0''\in g^{-1}(V)$ such that $f(x_0)=f(x_0'')$, and thus $g(x_0)=g(x_0'')\in V$ by \autoref{lem:rigidity_magic_contracts_fibers}. But as $g(x_0)\in \overline{\{g(\eta)\}}$ by continuity of $g$, this directly contradicts the way we defined $V$, so we are done.\qedclaim\medskip

    \noindent\textbf{Claim 2:} \textit{Let $F\colon S\to T$ be a morphism of varieties, and let $\xi\in S$ be a scheme point. Then $F\left(\overline{\{\xi\}}\cap S(\bC)\right)$ is dense in $\overline{\{F(\xi)\}}$.}\medskip
    
    \noindent\textit{Proof of Claim 2:} This directly follows from Chevalley's theorem and the fact that over an algebraically closed field, closed points are dense in any constructible set.\qedclaim\medskip
    
    \noindent\textbf{Claim 3:} \textit{Let $\eta$ be a scheme point of $X$. Then $f\left(g^{-1}\left(Y\setminus\overline{\{g(\eta)\}}\right)\right)$ is an open subset of $Y$.}\medskip
    
    \noindent\textit{Proof of Claim 3:}
    Let us start by showing that $g^{-1}(V)$ is saturated for $f$, i.e. that $g^{-1}(V)=f^{-1}(f(g^{-1}(V)))$. As both sets are constructible by Chevalley's theorem and constructible subsets of varieties over an algebraically closed field are determined by their closed points, it suffices to show that they contain the same closed points. So let $x\in f^{-1}(f(g^{-1}(V)))$ be a closed point. Then $f^{-1}(f(x))\cap g^{-1}(V)$ is a non--empty constructible set, so it contains a closed point $x'$. As then $f(x)=f(x')$, we obtain by \autoref{lem:rigidity_magic_contracts_fibers} that
    \begin{align*}
        g(x)=g(x')\in V,
    \end{align*}
    and thus $x\in g^{-1}(V)$. Hence we conclude that indeed $g^{-1}(V)=f^{-1}(f(g^{-1}(V)))$. Finally, as $f$ is a topological quotient map by \autoref{lem:rigidity_magic_top_quot}, the image under $f$ of a saturated open set is open in $Y$.\qedclaim\medskip

    Therefore, we may define a set-theoretic map $h\colon Y\to Z$ by imposing that $h(\theta)=g(\eta_\theta)$, where we choose $\eta_\theta\in f^{-1}(\theta)$ arbitrarily ($f$ is surjective by \autoref{lem:rigidity_magic_top_quot}). As $f$ is a topological quotient map, $h$ is in fact continuous, and by construction we have $g=h\circ f$.

    We are left to construct a map $h^{\#}\colon \cO_Z\to h_*\cO_Y$ making $(h,h^{\#})$ into a morphism of locally ringed spaces. As $f^{\#}$ is an isomorphism by \autoref{lem:rigidity_magic_fibration} and $g_*=h_*\circ f_*$ as functors, we may define
    \begin{align*}
        h^{\#}\coloneqq h_*f^{\#,-1}\circ g^{\#}.
    \end{align*}
    Then by construction, $(h,h^{\#})$ is a morphism of ringed spaces with $(g,g^{\#})=(h,h^{\#})\circ(f,f^{\#})$. It remains to see that $(h,h^{\#})$ is in fact a morphism of locally ringed spaces. This follows formally from the surjectivity of $f$ and the fact that $(g,g^{\#})$ and $(f,f^{\#})$ are morphisms of locally ringed spaces: let $\eta\in X$ be arbitrary, then we have the diagram
    %
    % https://q.uiver.app/#q=WzAsMyxbMCwwLCJcXHNPX3tZLGYoeCl9Il0sWzEsMCwiXFxzT197WixnKHgpfSJdLFsyLDAsIlxcc09fe1gseH0iXSxbMCwxLCJoX3tnKHgpfSJdLFsxLDIsImdfeCJdLFswLDIsImZfeCIsMCx7ImN1cnZlIjotM31dXQ==
\[\begin{tikzcd}
	{\cO_{Z,g(\eta)}} & {\cO_{Y,f(\eta)}} & {\cO_{X,\eta}}
	\arrow["{h_{f(\eta)}}", from=1-1, to=1-2]
	\arrow["{g_\eta}", bend left=30, from=1-1, to=1-3]
	\arrow["{f_\eta}", from=1-2, to=1-3]
\end{tikzcd}\]
    where both $f_\eta$ and $g_\eta$ are local. Therefore, we have
    \begin{align*}
        h_{f(\eta)}^{-1}(\fm_{Y,f(\eta)})=h_{f(\eta)}^{-1}(f_\eta^{-1}(\fm_{X,\eta}))=g_\eta^{-1}(\fm_{X,\eta})=\fm_{Z,g(\eta)}.
    \end{align*}
    As $f$ is surjective, we hence conclude that $h_\theta$ is local for every $\theta\in Y$. This concludes the proof.
\end{proof}

\begin{remark}\label{rem:alb_iso_rcc}
    Conceptually, and in view of the lack of properness assumptions in \autoref{lem:rigidity_magic}, an optimal version of \autoref{lem:alb_iso_rcc} would be that if $f\colon U\to W$ is surjective with quasi-rcc fibers and $W$ is normal, then $\Alb(f)$ is an isomorphism. However, we encountered technical difficulties when attempting to prove this, essentially because of the properness assumptions in Zariski's main theorem and Stein factorization. As the above version of \autoref{lem:alb_iso_rcc} is enough for our purposes, we didn't pursue this further.
\end{remark}

Finally, we generalize \autoref{lem:alb_iso_smooth} by showing that the same conclusion holds under the weaker hypothesis of having dlt singularities. 

\begin{lemma}\label{lem:alb_iso_dlt}
    Let $(X,\Delta)$ be a projective dlt pair and let $i\colon U\hookrightarrow X$ be an open immersion. Then $\Alb(i)$ is an isomorphism.
\end{lemma}

\begin{proof}
    Let $g\colon (Y,\Gamma)\to (X,\Delta)$ be a log resolution. By \cite[Corollary 1.6]{Hacon_Mckernan_On_Shokurovs_rational_connectedness_conjecture}, the fibers of $g$ are rationally chain connected. Let us denote $U'=g^{-1}(U)$ and $j\colon U'\to Y$. By \autoref{lem:alb_iso_smooth}, we have that $\Alb(j)$ is an isomorphism. On the other hand, by \autoref{lem:alb_iso_rcc}, both $\Alb(g)$ and $\Alb({g|_{U'}})$ are isomorphisms. Hence $\Alb(i)$ is an isomorphism as well.
\end{proof}

\subsection{The Bia\l ynicki--Birula decomposition}\label{subsec:Birula}

We will now use the Bia\l ynicki--Birula decomposition for torus actions, as constructed in \cite{Jelisiejew_Sienkiewicz_Bialynicki-Birula_decomposition_for_reductive_groups}, see \autoref{subsec:Albanese_strategy}. For a pair $(X,\Delta)$ satisfying the hypotheses of \autoref{thmA}, we thereby obtain an open subset of $X$ which is an affine fiber bundle over an open subset of $Z$. To make sure that $Z$ lies in an open cell of the decomposition, we need to construct a partial compactification $\overline{\bT}$ of $\bT$ so that for some closed point $z\in Z$, the action of $\bT$ on the cotangent space $T^{\vee}_z(X)$ extends to a $\overline{\bT}$--action. We do this using \autoref{lem:extend_lin_t_action}, but to apply it, the tangent space has to be of the right dimension. That is, we need a point $z\in X^{\sm}\cap Z$. Hence, we need to suppose that $(X,\Delta)$ is dlt, so that such a point exists by \autoref{rem:dlt_smooth_locus}.

\begin{lemma}\label{lem:BB-decomp}
    Let $(X,\Delta)$ be a dlt log Calabi--Yau pair and suppose there is a torus action $\bT\acts(X,\Delta)$ and a minimal lc center $Z$ with $\dim Z=c(\bT\acts X)$. Then for every $z\in X^{\sm}\cap Z$, there exists an open neighborhood $U\subseteq X^{\sm}$ containing $z$ and a diagram of the form
    %
    % https://q.uiver.app/#q=WzAsMyxbMCwxLCJcXHRVIl0sWzAsMCwiVSJdLFsxLDEsIlxcdFVcXHRpbWVzXFxiQV5lIl0sWzEsMiwiXFxjb25nIl0sWzEsMF0sWzAsMiwieFxcbWFwc3RvICh4LDApIiwyXSxbMCwxLCJcXHN1YnNldGVxIiwxLHsib2Zmc2V0IjotMywiY3VydmUiOi0yfV1d
\[\begin{tikzcd}
	U \\
	\tU & {\tU\times\bA^e}
	\arrow[from=1-1, to=2-1]
	\arrow["\cong", from=1-1, to=2-2]
	\arrow["\subseteq"{description}, shift left=3, bend left = 30, from=2-1, to=1-1]
	\arrow["{u\mapsto (u,0)}"', from=2-1, to=2-2]
\end{tikzcd}\]
    where $\tU=U\cap Z$ and $e=\codim Z$.
\end{lemma}

\begin{proof}
    As explained above, we want to construct $U$ by applying the Bia\l ynicki--Birula decomposition to the action of $\bT$ on $X$.  Note that $\bT$ acts on $X^{\sm}$ as automorphisms preserve the smooth locus. By \cite[Theorem 1.5]{Jelisiejew_Sienkiewicz_Bialynicki-Birula_decomposition_for_reductive_groups}, we have that $(X^{\sm})^{\bT}$ is smooth, and in particular it is the disjoint union of its irreducible components. By \autoref{cor:dim_bound_torus_center}, $Z$ is a top--dimensional component of $X^{\bT}$. So by \autoref{rem:dlt_smooth_locus}, $X^{\sm}\cap Z$ is a top--dimensional component of $(X^{\sm})^{\bT}$. Therefore, we have
    \begin{align*}
        T_z(X)^{\bT}=T_z(X^{\sm})^\bT\expl{=}{\cite[Theorem 5.2]{Fogarty_Fixed_point_schemes}} T_z((X^{\sm})^{\bT})= T_z(Z)
    \end{align*}
    and thus $\dim T_z(X)^\bT=\dim Z=c(\bT\acts X)$. So by \autoref{lem:dim_T_action}, we must have $\dim T_z(X)^\bT=c(\bT\acts X)$. Also, by Claim 2 in the proof of \autoref{lem:dim_T_action} we have $c(\bT\acts X)=c(\bT\acts T_z(X))$. It is then straightforward to see (e.g., by diagonalizing the action) that
    \begin{align*}
        \dim  T^{\vee}_z(X)^{\bT}=\dim  T_z(X)^{\bT}=c(\bT\acts T_z(X))=c(\bT\acts T^{\vee}_z(X)),
    \end{align*}
    so $T^{\vee}_z(X)$ satisfies the hypotheses of \autoref{lem:extend_lin_t_action}. Therefore, we obtain an isomorphism $\bT\cong\bG_m^e$ such that $\bG_m^e\acts T^{\vee}_z(X)$ has only non-negative weights. Hence, if under this isomorphism we compactify to $\overline{\bT}=\bA^e=(\bG_m\cup\{0\})^e$, then the $\bT$--action on $T^{\vee}_z(X)$ extends to a $\overline{\bT}$--action.
    
    We are now ready to apply the generalized Bia\l ynicki--Birula decomposition provided by \cite{Jelisiejew_Sienkiewicz_Bialynicki-Birula_decomposition_for_reductive_groups}. As by construction the $\bT$--action on $T^{\vee}_z (X)$ extends to a $\overline{\bT}$--action, \cite[Lemma 4.4]{Jelisiejew_Sienkiewicz_Bialynicki-Birula_decomposition_for_reductive_groups} gives that $T^{\vee}_z(X)$ has no outsider representations (see \cite[Definition 4.2]{Jelisiejew_Sienkiewicz_Bialynicki-Birula_decomposition_for_reductive_groups}). Hence, we can apply \cite[Proposition 1.6]{Jelisiejew_Sienkiewicz_Bialynicki-Birula_decomposition_for_reductive_groups} to $z\in X^{\sm}$, and thus obtain an open, affine and $\overline{\bT}$--stable neighborhood $U\subseteq X^{\sm}$ of $z$. By \cite[Proposition 5.7]{Jelisiejew_Sienkiewicz_Bialynicki-Birula_decomposition_for_reductive_groups} we then have in particular that $U^+=U$. Therefore we obtain the following diagram
    %
    % https://q.uiver.app/#q=WzAsNixbMCwwLCJYXntcXHJlZywrfSJdLFsxLDAsIlhee1xccmVnLFR9Il0sWzIsMCwiWF57XFxyZWd9Il0sWzIsMSwiWCJdLFsxLDEsIlheVCJdLFswLDEsIlheKyJdLFsyLDNdLFs0LDNdLFs1LDQsIlxccGlfWCIsMl0sWzAsNV0sWzEsNF0sWzEsMl0sWzAsMSwiXFxwaV97WF57XFxyZWd9fSJdXQ==
\[\begin{tikzcd}
	{U} & {U^{\bT}} & {U} \\
	{X^+} & {X^\bT} & X
	\arrow["{\pi_{U}}", from=1-1, to=1-2]
	\arrow[from=1-1, to=2-1]
	\arrow[from=1-2, to=1-3]
	\arrow[from=1-2, to=2-2]
	\arrow[from=1-3, to=2-3]
	\arrow["{\pi_X}"', from=2-1, to=2-2]
	\arrow[from=2-2, to=2-3]
\end{tikzcd}\]
    where the vertical maps are open immersions by \cite[Proposition 5.3]{Jelisiejew_Sienkiewicz_Bialynicki-Birula_decomposition_for_reductive_groups} (see also the list of properties preserved by $(-)^{+}$ after Proposition 1.6 in \emph{loc.cit.}). As $U$ is smooth, the map $\pi_U\colon U\to U^{\bT}$ has very nice properties by \cite[Theorem 1.5]{Jelisiejew_Sienkiewicz_Bialynicki-Birula_decomposition_for_reductive_groups}: $U^{\bT}$ is smooth and irreducible, and $\pi_U$ is an affine fiber bundle. As $U^{\bT}=X^{\bT}\cap U$ and as $X^{\sm}\cap Z$ is the only irreducible component of $(X^{\sm})^{\bT}$ containing $z$, we have $z\in U^{\bT}\subseteq Z$. Finally, as $\pi_U$ is an affine fiber bundle, there exists an open neighborhood $\tU\subseteq U^{\bT}$ of $z$ such that $\pi_U^{-1}(\tU)\cong\tU\times\bA^e$, compatible with $\pi_U$ and the projection. By replacing $U$ with $\pi_U^{-1}(\tU)$, we are done.
\end{proof}

\section{Proof of main theorem}
We now come to the proof of \autoref{thmA}. As we crucially needed $(X,\Delta)$ to be dlt to apply the Bia\l ynicki--Birula decomposition, we first prove it under this additional hypothesis. Moreover, we can give an explicit dense subset over which fibers are pairwise birational.

\begin{proposition}\label{prop:dlt_case_arbitrary_A}
    Let $(X,\Delta)$ be a dlt log Calabi--Yau pair and suppose there is a torus action $\bT\acts(X,\Delta)$ and a minimal lc center $Z$ with $\dim Z=c(\bT\acts X)$. Let $\alpha\colon X\to A$ be a fibration to an abelian variety $A$. Then $\alpha$ is generically birationally isotrivial.
\end{proposition}
\begin{proof}
    By \autoref{def:Albanese}, there exists $h\colon\Alb_X\to A$ such that $\alpha=h\circ\alb_X$.
    Note that by \cite[Theorem 3.3]{Bernasconi_Filipazzi_Patakfalvi_Tsakanikas_A_counterexample_to_the_log_canonical_Beauville-Bogomolov_decomposition} and \autoref{rem:dlt_smooth_locus}, $\alpha(X^{\sm}\cap Z)$ contains a dense open subset of $A$. Furthermore, the general fiber is irreducible (combining \stacksproj{038I} and \cite[E.1.(9)]{Gortz_Wedhorn_Algebraic_geometry_I}). Hence, there exists a dense open subset $V\subseteq A$ contained in $\alpha(X^{\sm}\cap Z)$ over which all closed fibers are integral.
    
    Let $a\in V$ be an arbitrary closed point and let $z\in X^{\sm}\cap Z$ be a preimage under $\alpha$. Then by \autoref{lem:BB-decomp}, we obtain an open neighborhood $U\subseteq X^{\sm}$ of $x$, such that we have an isomorphism $U\cong\tU\times\bA^e$, where $\tU=U\cap Z$ and $e=\codim Z$. We then obtain the following diagram:
    %
    % https://q.uiver.app/#q=WzAsOCxbMCwwLCJYIl0sWzEsMCwiVSJdLFsyLDAsIlxcd2lkZXRpbGRle1V9Il0sWzMsMCwiViJdLFszLDEsIlxcQWxiX1YiXSxbMiwxLCJcXEFsYl97XFx3aWRldGlsZGV7VX19Il0sWzEsMSwiXFxBbGJfVSJdLFswLDEsIlxcQWxiX1giXSxbMSwwLCIiLDAseyJzdHlsZSI6eyJ0YWlsIjp7Im5hbWUiOiJob29rIiwic2lkZSI6ImJvdHRvbSJ9fX1dLFsxLDJdLFsyLDMsIiIsMix7InN0eWxlIjp7InRhaWwiOnsibmFtZSI6Imhvb2siLCJzaWRlIjoidG9wIn19fV0sWzMsNF0sWzUsNCwiXFxzaW0iLDJdLFs2LDUsIlxcc2ltIiwyXSxbNiw3LCJcXHNpbSJdLFswLDddLFsxLDZdLFsyLDVdXQ==
\[\begin{tikzcd}
	X & U & {\widetilde{U}} & Z \\
	{\Alb_X} & {\Alb_U} & {\Alb_{\widetilde{U}}} & {\Alb_Z}
	\arrow[from=1-1, to=2-1]
	\arrow[hook', from=1-2, to=1-1]
	\arrow[from=1-2, to=1-3]
	\arrow[from=1-2, to=2-2]
	\arrow[hook, from=1-3, to=1-4]
	\arrow[from=1-3, to=2-3]
	\arrow[from=1-4, to=2-4]
	\arrow["\sim", from=2-2, to=2-1]
	\arrow["\sim"', from=2-2, to=2-3]
	\arrow["\sim"', from=2-3, to=2-4]
\end{tikzcd}\]
    The bottom outer arrows are isomorphisms by \autoref{lem:alb_iso_dlt} (note that $(Z,\Diff^*_Z\Delta)$ is klt, so in particular dlt), and the bottom middle arrow is an isomorphism by \autoref{lem:alb_iso_rcc} (as $U\to\widetilde{U}$ admits a section, namely the inclusion $\tU\subseteq U$). Therefore, we obtain a commutative diagram
    %
    % https://q.uiver.app/#q=WzAsNyxbMCwxLCJYIl0sWzEsMSwiVSJdLFszLDEsIlxcd2lkZXRpbGRle1V9Il0sWzQsMSwiWiJdLFsyLDIsIkIiXSxbMiwwLCJcXHdpZGV0aWxkZXtVfVxcdGltZXNcXGJBXmUiXSxbMiwzLCJBIl0sWzEsMCwiIiwwLHsic3R5bGUiOnsidGFpbCI6eyJuYW1lIjoiaG9vayIsInNpZGUiOiJib3R0b20ifX19XSxbMSwyXSxbMiwzLCIiLDAseyJzdHlsZSI6eyJ0YWlsIjp7Im5hbWUiOiJob29rIiwic2lkZSI6InRvcCJ9fX1dLFswLDQsIlxcYWxiX1giLDJdLFsxLDRdLFsyLDRdLFszLDQsIlxcYWxiX1YiXSxbMSw1LCJcXHNpbSJdLFs1LDIsIlxccHJfMSJdLFs0LDYsImgiXSxbMCw2LCJcXGFscGhhIiwyLHsiY3VydmUiOjN9XSxbMyw2LCJcXGJldGEiLDAseyJjdXJ2ZSI6LTN9XV0=
\[\begin{tikzcd}
	&& {\widetilde{U}\times\bA^e} \\
	X & U && {\widetilde{U}} & Z \\
	&& B \\
	&& A
	\arrow["{\pr_1}", from=1-3, to=2-4]
	\arrow["{\alb_X}"', from=2-1, to=3-3]
	\arrow["\alpha"', bend right = 20, from=2-1, to=4-3]
	\arrow["\sim", from=2-2, to=1-3]
	\arrow[hook', from=2-2, to=2-1]
	\arrow[from=2-2, to=2-4]
	\arrow[bend left = 10, from=2-2, to=3-3]
	\arrow[hook, from=2-4, to=2-5]
	\arrow[bend right = 10, from=2-4, to=3-3]
	\arrow["{\alb_Z}", from=2-5, to=3-3]
	\arrow["\beta", bend left = 20, from=2-5, to=4-3]
	\arrow["h", from=3-3, to=4-3]
\end{tikzcd}\]
    where $B$ is the Albanese of all the varieties of the top row. Then, from this diagram we obtain 
    \begin{align*}
        X_a\cap U\cong (Z_a\cap\widetilde{U})\times\bA^e,
    \end{align*}
    and note that $X_a\cap U$ and $Z_a\cap\tU$ are non-empty as they contain $z$. Therefore, we have $X_a\simeq Z_a\times\bP^e$. Now as $\alpha$ and $\alb_X$ are fibrations, we obtain that $h\colon B\to A$ is a fibration as well, and thus $\beta\coloneq h\circ\alb_Z\colon Z\to A$ is a fibration. Furhermore, by \autoref{lem:min_lc_fixed}.\autoref{lem:min_lc_fixed_stable}, the pair $(Z,\Diff^*_Z\Delta)$ is klt and $K$--trivial. Note that by \cite[Section 2.8.3]{Patakfalvi_Zdanowicz_On_the_Beauville-Bogomolov_decomposition_in_characteristic_p} and because the canonical divisor of $A$ is trivial, we have $K_{Z/A}=K_Z$ (see Section 2.8 in \emph{loc.cit.} for the definition of the relative canonical divisor $K_{Z/A}$ of $Z$ over $A$). So $K_{Z/A}+\Diff^*_Z\Delta$ is torsion. Hence, by \cite[Theorem 4.1]{Matsumura_Wang_Structure_theorem_for_projective_klt_pairs_with_nef_anti-canonical_divisor} (resp.\ \cite[Theorem A.13]{Patakfalvi_Zdanowicz_On_the_Beauville-Bogomolov_decomposition_in_characteristic_p} in the $\bQ$--factorial case) we obtain that $\beta$ is isotrival. Hence, if $a'\in V$ is another point, we have $Z_a\cong Z_{a'}$. Therefore, we obtain
    \begin{align*}
        X_a\simeq Z_a\times\bP^e\cong Z_{a'}\times\bP^e\simeq X_{a'},
    \end{align*}
    so we are done.
\end{proof}

We now come to the proof of \autoref{thmA}. At this point, all that is left to do is to pass to a dlt model with \autoref{lem:dlt_mod} and then apply \autoref{prop:dlt_case_arbitrary_A}.

\begin{proof}[Proof of \autoref{thmA}]
    Let $g\colon(X^{\dlt},\Delta^{\dlt})\to (X,\Delta)$ be the $\bT$--equivariant dlt model of $X$ provided by \autoref{lem:dlt_mod}. As in the proof of \autoref{lem:min_lc_fixed}.\autoref{lem:min_lc_fixed_invariant}, let $S$ be an lc center of $(X^{\dlt},\Delta^{\dlt})$ which is minimal for inclusion among lc centers dominating $Z$, and note that $S$ is actually a minimal lc center of $(X^{\dlt},\Delta^{\dlt})$. Notice that we have the chain of inequalities
    \begin{align*}
        \dim Z\leq\dim S\explshift{-1.5em}{\leq}{\autoref{lem:min_lc_fixed}.\autoref{lem:min_lc_fixed_invariant}}\dim X^{\dlt,\bT}\expl{\leq}{\autoref{lem:dim_T_action}} c(\bT\acts X^{\dlt}) \expl{=}{\autoref{lem:dlt_mod}.\autoref{lem:dlt_mod_T-equivariant}} c(\bT\acts X) = \dim Z
    \end{align*}
    and thus we must have $\dim S=\dim Z$. In particular, the lc center $S$ of the pair $(X^{\dlt},\Delta^{\dlt})$ satisfies the hypotheses of \autoref{thmA} for the action $\bT\acts(X^{\dlt},\Delta^{\dlt})$. Furthermore, as $g_*\cO_{X^{\dlt}}=\cO_X$, the map $\alpha=\alb_X\circ g$ is a fibration of $X^{\dlt}$ to the abelian variety $\Alb_X$. By \autoref{prop:dlt_case_arbitrary_A}, we hence obtain that $\alpha$ is generically birationally isotrivial. Finally, as $g$ is birational and $\alpha=\alb_X\circ g$, we obtain that also $\alb_X$ is generically birationally isotrivial.
\end{proof}

\begin{remark}\label{rem:thmA_general_A}
    By the same proof, \autoref{thmA} holds for any fibration to an abelian variety.
\end{remark}

\section{Examples}

In this section we construct an example satisfying the hypotheses of \autoref{thmA} for which the Albanese morphism is neither generically isotrivial, nor birationally isotrivial. To warm up, we start with a simple example where the Albanese morphism is generically isotrivial but not birationally isotrivial.
\begin{example}[\protect{\cite[Example 3.1]{Bernasconi_Filipazzi_Patakfalvi_Tsakanikas_A_counterexample_to_the_log_canonical_Beauville-Bogomolov_decomposition}}]\label{ex:gen_isotr_not_isotr}
    Consider the pair
    \begin{align*}
        (X,\Delta)\coloneqq (E\times\bP^1,E\times\{0\}+E\times\{\infty\}),
    \end{align*}
    where $E$ is an elliptic curve, and denote by $\pi_E\colon X\to E$ the projection. Let $f\colon Y\to X$ be the blow-up of a closed point $x\in E\times\{0\}$, and let $\Delta_Y$ be the strict transform of $\Delta$, so that $(Y,\Delta_Y)$ and $(X,\Delta)$ are crepant birational. Furthermore, we have $\alb_Y=\pi_E\circ f$.

    Finally, note that the obvious action $\bG_m\acts E\times\bP^1$, i.e. acting on the second factor by scaling, induces an action $\bG_m\acts (X,\Delta)$. Furthermore, the action lifts to $(Y,\Delta_Y)$, as we blew up a fixed point. The general orbit is one--dimensional, so the action has complexity $1$. As the components of $(Y,\Delta_Y)$ are minimal lc centers, the hypotheses of \autoref{thmA} are fulfilled.

    Now let us analyze the fibers of $\alb_Y$: over $t\neq\pi_E(x)$, the fiber $Y_t$ is isomorphic to $\bP^1$, and in particular, $\alb_Y$ is generically isotrivial. However, the fiber over $t=\pi_E(x)$ is isomorphic to two projective lines intersecting transversely at a point, so $\alb_Y$ is not birationally isotrivial.
\end{example}

Let us now construct an example where the Albanese morphism is not birationally isotrivial and not generically isotrivial. By the following remark, it has to be of dimension greater than or equal to $3$.

\begin{remark}\label{rem:alb_gen_isotr}
    If $(X,\Delta)$ is a log Calabi--Yau pair with an effective torus action $\bT\acts (X,\Delta)$ such that $\dim\bT+\dim\Alb_X=\dim X$, then $\alb_X$ is generically isotrivial. Indeed, as the induced action $\bT\acts\Alb_X$ is trivial, $\bT$ acts on the fibers of $\alb_X$. The fibers are of dimension $\dim\bT$, and as the general orbit of $\bT\acts X$ has codimension $c(\bT\acts X)=\dim X-\dim\bT$, we obtain that a general fiber of $\alb_X$ is toric. Let $\{T_i\}_{i\in\bN}$ be an enumeration of the projective toric varieties of dimension $\dim\bT$ up to isomorphism. Then the subset $S_i=\{a\in\Alb_X\mid X_a\cong T_i\}$ is constructible (as it is the image of $\Isom_{\Alb_X}(X,T_i\times_{\bC}\Alb_X)\to\Alb_X$), and $\bigcup_i S_i$ contains a non-empty open subset of $\Alb_X$. Hence, one of the $S_i$ must be dense, and thus itself contains a non-empty open subset of $\Alb_X$. That is, the general fiber of $\alb_X$ is isomorphic to $T_i$. In particular, if $X$ is a surface, then for $\alb_X$ not to be isotrivial we must have $\dim\Alb_X=1$, and so if we have a non-trivial action $\bG_m\acts X$, $\alb_X$ is generically isotrivial.
\end{remark}

\begin{example}\label{ex:bir_isotr_not_gen_isotr}
    The construction is similar to \autoref{ex:gen_isotr_not_isotr}, but we replace the elliptic curve by a bielliptic surface $S$. More precisely, in the terminology of \cite[Proposition 1.2]{Serrano_Divisors_of_bielliptic_surfaces_and_embeddings_in_P4}, let $S$ be of Type 1, i.e., 
    \begin{align*}
        S=\factor{A\times B}{G},
    \end{align*}
    where $A$ and $B$ are elliptic curves, $G=\bZ/2\bZ$, acting on $A$ by translations (i.e., $a\mapsto a+P$ for a fixed $2$--torsion point $P\in A$) and on $B$ by symmetry (i.e., $\iota\colon b\mapsto -b$). Furthermore, we take $B$ to be general, so that $\Aut(B)$ is generated by translations and $\iota$ (e.g., by \cite[Theorem VI.16]{Beauville_Complex_algebraic_surfaces}). 
    
    Note that $S$ admits two elliptic fibrations: on the one hand, the projection $A\times B\to A$ induces an elliptic fibration
    \begin{align*}
        \Phi\colon S\to A'\coloneqq A/\langle a\mapsto a+P\rangle,
    \end{align*}
    whose fibers are isomorphic to $B$, which is the Albanese morphism of $S$. On the other hand, the projection $A\times B\to B$ induces an elliptic fibration
    \begin{align*}
        \Psi\colon S\to \bP^1\cong B/\langle b\mapsto -b\rangle,
    \end{align*}
    whose smooth fibers are isomorphic to $A$.

    Now consider the log Calabi--Yau pair
    \begin{align*}
        (X,\Delta)\coloneqq (S\times\bP^1,S\times\{0\}+S\times\{\infty\}),
    \end{align*}
    and denote by $\pi_S\colon X\to S$ the projection to $S$. Note that the Albanese morphism of $X$ is just the composition $\Phi\circ\pi_S$, and so the fiber $X_t$ over any $t\in A'$ is just $S_t\times\bP^1\cong B\times\bP^1$. The idea is to blow up $X$ at some curves in $S\times\{0\}$ to break the isotriviality.

    Let $A_0'$ be the half fiber of $\Psi$ over the point $0\cdot G$ in $B/G=\bP^1$, so that $A_0'$ is a section of $\Phi$ (it is the image of $A\times\{0\}$ under the quotient map $A\times B\to S$). Let $C$ be a smooth irreducible curve in $S$ intersecting $A_0'$ transversely and which dominates both $A'$ and $\bP^1$. Then the strict transform $\overline{A_0'}$ of $A_0'\times\{0\}$ in $\Bl_{C\times\{0\}}X$ is smooth, and so
    \begin{align*}
        f\colon \underbrace{\Bl_{\overline{A_0'}}\Bl_{C\times\{0\}}X}_{Y\coloneqq}\to X
    \end{align*}
    is smooth. Let $\Delta_Y$ be the strict transform of $\Delta$ in $Y$. Then the pair $(Y,\Delta_Y)$ is crepant birational to $(X,\Delta)$, and in particular it is log Calabi--Yau. Furthermore, we have $\alb_Y=\Phi\circ \pi_S\circ f$, as the fibers of $f$ and $\pi_S$ are rationally chain connected.

    Let us prove that $\alb_Y$ is not generically isotrivial. For $t\in A'$, we denote the point of intersection of $A_0'$ and $S_t$ by $0$, and we choose a group structure on $S_t$ with origin $0$. Note that for general $t\in A'$, $C_t$ consists of $n$ distinct points in $S_t\setminus\{0\}$. Hence, the fiber $Y_t$ of $\alb_Y$ over $t$ is given by
    \begin{align*}
        Y_t=\Bl_{(\{0\}\cup C_t)\times\{0\}} (S_t\times\bP^1).
    \end{align*}
    Suppose now that for general $t,t'\in A'$ we have an isomorphism $Y_t\cong Y_{t'}$. As $S_t\times\bP^1$ and $S_{t'}\times\bP^1$ are both isomorphic to $B\times\bP^1$, which is in particular minimal, the above isomorphism induces an isomorphism $S_t\times\bP^1\cong S_{t'}\times\bP^1$ which restricts to a bijection between $(\{0\}\cup C_t)\times\{0\}$ and $(\{0\}\cup C_{t'})\times\{0\}$. Note also that we have
    \begin{align*}
        \Aut(B\times\bP^1)=\Aut(B)\times\Aut(\bP^1).
    \end{align*}
    Indeed, this follows from \cite[Theorem 2.(5)]{Maruyama_On_automorphism_groups_of_ruled_surfaces}, by observing that the natural map $\Aut(B)\to \Aut(B\times\bP^1)$ provides a section of the exact sequence given by \cite[Lemma 6]{Maruyama_On_automorphism_groups_of_ruled_surfaces}. Therefore, the isomorphism $S_t\times\bP^1\cong S_{t'}\times\bP^1$ induces an isomorphism $S_t\cong S_{t'}$ under which $\{0\}\cup C_t$ is mapped to $\{0\}\cup C_{t'}$.

    Write now $\{0\}\cup C_t=\{p_{0,t},\ldots,p_{n,t}\}$, and consider the effective divisor
    \begin{align*}
        D_t\coloneqq\sum_{i>j}[\Psi(p_{i,t}-p_{j,t})]
    \end{align*}
    on $\bP^1$ (here $p_{i,t}-p_{j,t}$ denotes the difference of the two points in the group structure chosen on $S_t$). As we took $B$ so that $\Aut(B)$ is generated by translations and inversion, observe that the map
    \begin{align*}
        B\times B&\to \bP^1=B/\langle b\mapsto -b\rangle\\
        (b,b') &\mapsto \pm(b-b')
    \end{align*}
    is invariant under the diagonal action of $\Aut(B)$. Therefore, the existence of an isomorphism $S_t\cong S_{t'}$ inducing a bijection between $\{p_{i,t}\}$ and $\{p_{i,t'}\}$ forces $D_t=D_{t'}$ for general $t,t'\in A'$ (as $\Psi|_{S_t}$ is the quotient map of $S_t$ under the action by inversion). But then as $A'$ is a curve, this implies that the set $\{D_t\}_{t\in A'}$ is finite. As furthermore $\Psi(C_t)\subseteq\Supp D_t$, we  obtain that $\Psi(C)$ is finite, which is a contradiction. Hence $\alb_Y$ is not generically isotrivial.
    
    On the other hand, note that the action $\bG_m\acts (X,\Delta)$ on the second factor lifts to $Y$, as the centers of the blow--ups are in the fixed locus. As the action is effective, it has complexity $2$. Furthermore, the components of $\Delta_Y$ are disjoint, so they are minimal lc centers. Hence the hypotheses of \autoref{thmA} are fullfilled.

    Note however that $\alb_Y$ is birationally isotrivial, as all fibers are birational to $B\times\bP^1$. To obtain an example which is not birationally isotrivial, we proceed in a similar fashion as in \autoref{ex:gen_isotr_not_isotr}: choose some $x\in A'=\Alb_S$, then the fiber $S_x$ is isomorphic to $B$, so in particular it is smooth. Consider
    \begin{align*}
        g\colon \underbrace{\Bl_{S_x\times\{\infty\}} Y}_{Z\coloneqq}\to Y
    \end{align*}
    and let $E$ be the exceptional divisor. Then we have $\alb_Z=\alb_Y\circ g$, and the fiber over $x$ is the union of $Y_x$ and $E$, so it is not irreducible. As all other fibers are irreducible, we conclude that $\alb_Z$ is not birationally isotrivial. For $t\neq x$, we have $Z_t\cong Y_t$, and thus $Z$ is not generically isotrivial.

    Finally, if $\Delta_Z$ is the strict transform of $\Delta_Y$, it is straightforward to check that $(Z,\Delta_Z)$ still satisfies the hypotheses of \autoref{thmA}, as it is crepant birational to $(Y,\Delta_Y)$ and we blew--up a curve in the fixed locus of $\bG_m\acts Y$.
\end{example}

\bibliographystyle{amsalpha}
\bibliography{bibliography}
\end{document}